\documentclass{amsart}
\usepackage{amssymb,latexsym,graphicx, amscd}
\newtheorem{theorem}{Theorem}[section]
\newtheorem{prop}[theorem]{Proposition}
\newtheorem{lemma}[theorem]{Lemma}
\newtheorem{remark}[theorem]{Remark}

\newtheorem{cor}[theorem]{Corollary}

\newtheorem{conj}[theorem]{Conjecture}

\newcommand{\na}{\nabla}

\newcommand{\La}{\Lambda}

\newcommand{\ka}{K{\"a}hler }

\newcommand{\leftr}{[\hbox{\hspace{-0.15em}}[}
\newcommand{\rightr}{]\hbox{\hspace{-0.15em}}]}

\begin{document}
\title{On a class of almost Hermitian 4-manifolds}

\author{Ethan Addison}
\address{Department of Mathematics, Texas State University, San Marcos, TX 78666-4684, USA}
\email{el.addison@txstate.edu}

\author{Tedi Dr\u{a}ghici} 
\address{Department of Mathematics and Statistics, Florida International University, Miami, FL 33199, USA}
\email{draghici@fiu.edu}

\author{Mehdi Lejmi}
\address{Department of Mathematics, Bronx Community College of CUNY, Bronx, NY 10453, USA.}
\email{mehdi.lejmi@bcc.cuny.edu}

%
\begin{abstract} Using an integral identity proved by Sekigawa \cite{Sek87} on compact almost Hermitian 4-manifolds, we naturally obtain a global characterization of the class $\mathcal{AH}_1$ of almost Hermitian 4-manifolds satisfying the first Gray curvature condition from apparently weaker conditions. Then we take steps towards a classification of almost Hermitian 4-manifolds of class $\mathcal{AH}_1$, including proving a uniqueness result on 4-dimensional Lie algebras.
\end{abstract}

\maketitle

\section{Introduction} 
Let $(M^{2n}, g, J, \omega)$ be an almost Hermitian manifold, that is, $J$ is an almost complex structure, $J^2 = - Id$, orthogonal with respect to the Riemannian metric $g$, and $\omega$ is the non-degenerate 2-form given by $\omega(X,Y) = g(JX,Y)$, for any $X,Y \in TM$. 
The best almost Hermitian structures are the {\it K\"ahler structures}, that is, those for which $J$ comes from a complex atlas of coordinates on the manifold and $\omega$ is a symplectic form, i.e. $d\omega = 0$. 
K\"ahler manifolds are also characterized by the fact that both $J$ and $\omega$ are parallel with respect to the Levi-Civita connection $\nabla$ induced by the metric, and this leads to many consequences, including topological restrictions for the existence of compact K\"ahler manifolds.

From the 1970s and probably even earlier, there has been interest in studying almost Hermitian structures whose curvature resembles that of a K\"ahler manifold. In particular, Alfred Gray, \cite{Gr}, introduced the following first three conditions on the Riemannian curvature tensor $R$ of an almost Hermitian manifold, while the fourth one, on the Ricci tensor, gained popularity starting with the work of Blair and Ianus \cite{BI} (but surely Gray was also aware of it).
\begin{eqnarray} \nonumber
(G_1) & & R_{XYZW} = R_{XYJZJW} \; ; \\ \nonumber
(G_2) & & R_{XYZW} - R_{JXJYZW} = R_{JXYJZW} + R_{JXYZJW} \; ; \\ \nonumber
(G_3) & & R_{XYZW} = R_{JXJYJZJW} \; ; \\ \nonumber  
(G_4) & & {\rm Ric}_{XY} = {\rm Ric}_{JXJY} \; .
\end{eqnarray}

We will denote by $\mathcal{AH}$ the class of all almost Hermitian manifolds and by $\mathcal{AH}_{i}$ the class of almost Hermitian manifolds satisfying the $i$-th Gray condition $(G_i)$. Simple applications of the first Bianchi identity and a contraction yield the implications $(G_1) \Rightarrow (G_2) \Rightarrow (G_3) \Rightarrow (G_4)$. It is also well known that K\"ahler manifolds satisfy the property $(G_1)$, hence all the others. If we also denote by $\mathcal{K}$ the class of K\"ahler manifolds, we have the obvious inclusions
$$ \mathcal{K}  \subseteq \mathcal{AH}_1 \subseteq \mathcal{AH}_2 \subseteq \mathcal{AH}_3 \subseteq \mathcal{AH}_4 \subseteq  \mathcal{AH}\; .$$
In fact, if the dimension is at least 4, most of the inclusions above are known to be strict even if the manifold is assumed compact. For instance, on the 4-torus $\mathbb{T}^4$ with a flat metric $g$, there are (infinitely many) non-integrable almost complex structures $J$ compatible with the metric, hence such a structure $(\mathbb{T}^4,g,J,\omega)$ is not K\"ahler, but it is trivially of class $\mathcal{AH}_{1}$, as the curvature tensor is identically zero. 

In dimension 4, which is the focus of this paper, the only special classes of almost Hermitian manifolds in terms of the Gray-Hervella classification \cite{GrHe} are the Hermitian manifolds ($N=0$) and the almost K\"ahler manifolds ($d\omega = 0$). We denote by $\mathcal{H}$ and $\mathcal{AK}$ the classes of Hermitian manifolds, and, respectively, almost K\"ahler manifolds, and by $\mathcal{H}_i$, $\mathcal{AK}_i$ the corresponding subclasses which satisfy the $i$-th Gray condition $(G_i)$.
It was observed in \cite{Go} that the equality $\mathcal{AK}_1 = \mathcal{K}$ holds locally in all dimensions. In \cite{DM} examples of compact non-K\"ahler manifolds of class $\mathcal{AK}_2$ were found in dimension 6, and, hence, in any higher dimension by taking products with compact K\"ahler manifolds. By contrast, in dimension 4, it was shown in \cite{AD-QJM}, \cite{AAD-AGAG}, \cite{AAD} that in the compact case the equalities $\mathcal{AK}_3=\mathcal{AK}_2= \mathcal{K}$ hold. Locally, the inclusions $\mathcal{K} \subset \mathcal{AK}_2 \subset\mathcal{AK}_3 $ are strict in dimension 4, but a complete classification of all possible examples was found in the above mentioned papers. The geometric structure of all these local examples is surprisingly rich. For the case of Hermitian surfaces, it is known that in the compact case the equality $\mathcal{H}_1 = \mathcal{K}$ holds (see also Remark \ref{AK1orH1=K}). It is also known that even locally, the condition that the Ricci tensor is $J$-invariant on a Hermitian surface implies that the manifold satisfies the second Gray condition (see e.g. \cite{AG}). That is, the equality $\mathcal{H}_4= \mathcal{H}_2$ holds locally in dimension 4. A description of compact non-K\"ahler Hermitian surfaces with $J$-invariant Ricci tensor is known when the first Betti number $b_1$ is even \cite{AG}, but only partial results are available when $b_1$ is odd, (e.g. see \cite{Mu}).

At this point, let us make some comments about the special role that condition $(G_4)$ appears to have. In \cite{BI} it was observed that the Hilbert functional
$$H(g) = \int_M s_g \; \mu_g \; $$
when restricted on the space of almost Hermitian structures with a fixed fundamental form $\omega$ on a compact manifold has critical points precisely the almost Hermitian structures satisfying the condition $(G_4)$. In fact, Blair and Ianus did this in the almost K\"ahler case, that is, when $\omega$ is a fixed symplectic form, but one can easily see that the critical points are the same even when $\omega$ is not closed. Only in the almost K\"ahler case, however, the Hilbert functional is bounded from above by a symplectic invariant (see \cite{Bl})
$$H((g, J, \omega)) = \int_M s_g \; \frac{\omega^n}{n!} \leq \frac{1}{4\pi} c_1\cup [\omega^{n-1}] \; ,$$
with equality if and only if $(g, J, \omega)$ is a K\"ahler structure. This observation, together with a famous still open conjecture of Goldberg \cite{Go} stating that a compact almost K\"ahler Einstein manifold must be K\"ahler Einstein, motivated a question of Blair and Ianus whether compact almost K\"ahler manifolds with $J$-invariant Ricci tensor are necessarily K\"ahler. This turned out to have a negative answer in dimensions 6 and higher, due to the examples from \cite{DM}. However, in dimension 4, despite some positive partial results, the question of Blair and Ianus remains open. As we have already mentioned Goldberg conjecture, let us add that the most important partial result is Sekigawa's theorem (see \cite{sekigawa}) confirming the conjecture when the scalar curvature is non-negative. His result is a consequence of an integral identity he proved using Chern-Weil theory on a compact almost K\"ahler manifold of an arbitrary dimension.

In Section 3, we will use a related but less known integral identity, also due to Sekigawa \cite{Sek87}, which is valid on compact 4-dimensional almost Hermitian manifolds and whose proof we give for completeness. We then show (see Proposition \ref{ah1}) that we naturally arrive at the class $\mathcal{AH}_{1}$ of almost Hermitian manifolds satisfying the first Gray condition $(G_1)$ from the assumption that the Ricci tensor is $J$-invariant (condition $(G_4)$) and one additional global assumption. In Section 5, we make first steps towards a classification of the 4-dimensional manifolds of class $\mathcal{AH}_{1}$. Although we cannot complete this classification at this time, we gather enough evidence to at least make a conjecture about the compact case at the beginning of Section 5. In that section, we give several partial results confirming the conjecture under some additional assumptions. In Section 6, we show that there is a unique 4-dimensional Lie algebra that admits an invariant non-K\"ahler $\mathcal{AH}_{1}$-structure. Section 2 contains some preliminary results, some of which may be of independent interest (see, for example, Proposition \ref{Phi=dJth}). Section 4 is the technical core of our paper, as we extract, in terms of 2-forms and the $U(2)$-decomposition of 2-forms, information provided by the differential Bianchi identity on an arbitrary almost Hermitian 4-manifold. The lemmas in Section 4 extend similar results from \cite{AAD} obtained for almost K\"ahler 4-manifolds.
Under the $(G_1)$-assumption, the lemmas in Section 4 are crucial for the results obtained in Sections 5 and 6.

\vspace{0.2cm}

\noindent
Here is a theorem obtained from combining results of Sections 3 and 5. Its proof is given in Section 5. The star-Ricci form that appears in the statement is defined by $\rho^* = R(\omega)$.

\begin{theorem} \label{Eah1}
Let $(M^4, g, J, \omega)$ be a compact Einstein almost Hermitian 4-manifold. Then the following inequality is satisfied
$$\int_M \rho^* \wedge \rho^* \leq \frac{1}{4\pi^2} c_1^2(M) \; ,$$ 
with equality if and only if the manifold is K\"ahler-Einstein, or if the metric is Ricci flat and anti-self-dual ($W^+ = 0$). In the second case, $(M^4, g)$ is the 4-torus (or a quotient of the torus) with a flat metric, or a K3 surface (or a quotient of a K3 surface) with a Ricci flat K\"ahler metric. Still in this second case, $J$ is any $g$-compatible almost complex structure, not necessarily integrable.
\end{theorem}

\noindent
Modulo the conjecture from the beginning of Section 5, we expect that the Einstein assumption could be weakened to $J$-invariant Ricci tensor, and that the only non-K\"ahler examples are still the ones from the statement.


\section{Preliminaries} 
We will generally follow the conventions and notations of \cite{AAD},  \cite{AAD-AGAG}, \cite{AD-QJM}. In particular, we refer the reader to Section 2 of \cite{AAD} for more details on the preliminary material. 
Throughout the paper, $(M^4, g, J, \omega)$ will denote an almost Hermitian  manifold of (real) dimension $4$, where $J$ is a $g$-orthogonal almost-complex structure for the Riemannian metric $g$, and $\omega(\cdot, \cdot)= g(J\cdot,\cdot)$ is the induced fundamental 2-form. It is well known that the covariant derivative of $\omega$ with respect to the Levi-Civita connection $\nabla$ is given in dimension 4 by
\begin{equation} \label{nablaom-dim4}
\nabla_X \omega = \frac{1}{2} \big( X^{\flat} \wedge J\theta + JX^{\flat} \wedge \theta \big) + \frac{1}{2} N_{JX}  \; ,
\end{equation}
where $N_{\cdot} \in \Lambda^1 M \otimes \Lambda^2 M$ is essentially the Nijenhuis tensor:
\begin{equation} \label{Nij-def}
 N_X(A,B) = \langle N(A,B), X \rangle = \langle [JA,JB] - [A, B] - J[JA, B] - J[A, JB],   X \rangle  \; ,    
\end{equation}
and $\theta$ is the Lee 1-form of the structure defined by
\begin{equation} \label{theta-def}
 \theta = J\delta \omega \; , \mbox{ or equivalently by } \; d\omega = \theta \wedge \omega \; .
\end{equation}
In the above equation and throughout the paper, we will use the extension of $J$ to the bundle of (real) 1-forms, $\Lambda^1 M$, 
$(J\alpha)(X) = -\alpha(JX)$, for any $\alpha \in \Lambda^1 M$, so that $J$ commutes with the Riemannian duality between $TM$ and $\Lambda^1 M $. We will denote the inner product induced by the metric $g$ on various bundles of forms and tensors on $M$, including sometimes even on $TM$, by $\langle \cdot,\cdot \rangle$.
It is well known that the almost complex structure $J$ gives rise to a type decomposition of complex vectors and forms. We will work mainly with real vectors and forms, and, in particular, the following $U(2)$-decompositions of the bundle of real two-forms $\Lambda^2 M$ 
will be used often:
\begin{equation}\label{Lambda2r}
\La^2M = {\mathbb R}\cdot \omega\oplus \La^{1,1}_0 M \oplus \leftr
\La^{0,2}M
\rightr \; .
\end{equation}
As in \cite{AAD}, we will use the superscript $'$ to denote the $J$-invariant part of a 2-form (or a 2-tensor), and the superscript $''$ for the $J$-anti-invariant part, while the subscript $_0$ denotes the trace-free part. Thus, if $\psi \in \La^2 M$,
$$ \psi = \psi' + \psi'' = \frac{1}{2} \langle \psi, \omega \rangle \omega +\psi'_0 + \psi'' \; , $$
where
$$\psi'(\cdot , \cdot) = \frac{1}{2}(\psi(\cdot , \cdot) + \psi(J\cdot,
J\cdot)) \; , \;
\psi''(\cdot , \cdot) = \frac{1}{2}(\psi(\cdot , \cdot) - \psi(J\cdot,
J\cdot))  \; \mbox{ and } $$
$$ \psi_0 = \psi - \frac{1}{2} \langle \psi, \omega \rangle \omega \; .$$ 
The decomposition (\ref{Lambda2r}) can thought of as a refinement of the self-dual, anti-self-dual decomposition of two-forms induced by the Hodge operator $\star_g$ in dimension 4 
\begin{equation*}\label{SO(4)La2}
\La^2 M = \La^+ M \oplus \La^-M \; .
\end{equation*}
 In fact, we have
\begin{equation}\label{U(2)La+}
\La^+M = {\mathbb R}\cdot \omega \oplus \leftr
\La^{0,2}M \rightr, \ \ \
\La^-M = \La^{1,1}_0 M \; .
\end{equation}
Note that the bundle $\leftr \La^{0,2}M \rightr$ is the real underlying bundle of the anti-canonical bundle $ \La^{0,2}M $ and we still denote by $J$ the induced complex structure on $\leftr \La^{0,2}M \rightr$ acting by
$$ (J\psi)(X,Y) =- \psi(JX, Y) \; , \; \; \forall \; \psi \in \leftr \La^{0,2}M \rightr \; .$$
In the local computations that follow, we will often choose a local section $\phi$ of $\leftr \La^{0,2}M \rightr$ 
such that $| \phi |^2 = 2$. As there is an $S^1$ freedom for its choice, we sometimes refer to $\phi$ as a {\it gauge}. Note that $\{ \phi, J\phi \}$ determines a local frame for $\leftr \La^{0,2}M \rightr$, and $\{\omega, \phi, J\phi \}$ is a frame for $\Lambda^+M$. With the choice of gauge $\phi$, there are local $1$-forms $a, b, c$ so that
\begin{equation} \label{abc}
\nabla \omega = a \otimes \phi + b \otimes J\phi \; , \; \nabla \phi = - a \otimes \omega  + c \otimes J\phi \; , \;
\nabla J\phi = - b\otimes \omega - c \otimes \phi \; .
\end{equation}
There is also a local 1-form $n$ that, along with $Jn$, locally determines the Nijenhuis tensor in the frame $\{ \phi, J\phi \}$:
\begin{equation} \label{NphiJphi}
    N_{JX} = n(X) \phi - (Jn)(X) J\phi \; , \; \; N_X = (Jn)(X) \phi + n(X) J \phi \; .
\end{equation}
\noindent 
Note that above formulas reflect the property $N_{JX} = -JN_X$, which can also be checked from the definition of the Nijenhuis tensor.
From (\ref{nablaom-dim4}), one also checks that the various 1-forms are related by
\begin{equation} \label{ab-thetan}
  J\phi( \theta) = a - Jb \; , \; \; n = a+ J b \; .
\end{equation}

\subsection{The $U(2)$-decomposition of curvature and the condition $(G_1)$.} With respect to the decomposition (\ref{Lambda2r}), the curvature operator as an element of $S^2(\Lambda^2 M)$ (still denoted by $R$, as the curvature tensor) decomposes as \cite{TV}
\begin{equation}\label{u(2)}
R = \frac{s}{12} {\rm Id}_{| \La^2M} + W_1 ^+ + W_2 ^+ + W_3 ^+
 + {\widetilde{{\rm Ric}'_0}}
 + {\widetilde {{\rm Ric}''_0}} + W^- ,
\end{equation}
where the components are as follows:

-- $W^+_1$ denotes the ``scalar''-component of $W^+$, determined by the conformal scalar curvature $\kappa$ by
\begin{equation}\label{w^+1}
W_1^+ = \frac{\kappa}{8} \omega \otimes \omega - \frac{\kappa}{12} {\rm
Id}_{| \La^+M} \; , \; \; \;  \kappa = 3 \langle W^+(\omega), \omega \rangle = \frac{3s^* - s}{2} \; ;
\end{equation}

-- $W^+_2$ is the component of $W^+$ that interchanges factors ${\mathbb R}\cdot \omega \oplus \leftr \La^{0,2}M \rightr $ by
\begin{equation}\label{w^+2}
W^+_2 = \frac{1}{2}\big( {\rho^*}'' \otimes \omega + \omega \otimes  {\rho^*}'' \big) \; ;
\end{equation}

-- $W^+_3$ is the component of $W^+$ that acts on $\leftr \La^{0,2}M \rightr $ but anti-commutes with the action of $J$ on this bundle - specifically, for some locally defined smooth functions $\alpha$ and $\beta$
\begin{equation}\label{w^+3} 
W^+_3 = \frac{\alpha}{2} [ \phi \otimes \phi - J\phi \otimes J\phi ] + 
\frac{\beta}{2} [ \phi \otimes J\phi + J\phi \otimes \phi] \; ; 
\end{equation}

-- As the notation indicates, ${\widetilde{{\rm Ric}'_0}}$ is the component of the curvature operator determined by the trace-free $J$-invariant part of the Ricci tensor; this interchanges the components 
${\mathbb R}\cdot \omega$ and $\La^-M = \La^{1,1}_0 M$ of decomposition (\ref{Lambda2r}).

-- ${\widetilde{{\rm Ric}''_0}}$ is the component of the curvature operator determined by the trace-free $J$-anti-invariant part of the Ricci tensor; this interchanges the components 
$\leftr \La^{0,2}M \rightr$ and $\La^-M = \La^{1,1}_0 M$ of decomposition (\ref{Lambda2r}).

-- $W^-$ is the anti-self-dual part of $W$ and acts on the $\La^-M$ component.

\vspace{0.2cm}

\noindent Let us note that between these various curvature components there exist further relations determined by the differential Bianchi identity. We explore these systematically in Section 4.

\vspace{0.2cm}

\noindent
We will spend a few more words here about the Ricci forms we will use in the paper. In the introduction, we defined the star-Ricci form $\rho^* = R(\omega)$; the star-Ricci tensor ${\rm Ric}^*$ is defined by
$${\rm Ric}^*(X,Y) = - R(\omega)(JX, Y) = - \frac{1}{2} \langle R(JX, Y)e_i \, , \, Je_i \rangle \; , $$
where $\{ e_i \}$ is an orthonormal basis of $TM$. Note that in general the star-Ricci tensor is not symmetric, but satisfies ${\rm Ric}^*(JX,JY) = {\rm Ric}^*(Y,X)$. For an arbitrary almost Hermitian manifold, we define the Ricci form $\rho$, using the $J$-invariant part ${\rm Ric}'$ of the Ricci tensor, by
$$ \rho(X,Y) = {\rm Ric}'(JX, Y) \; ,$$
so observe that $\rho$ is, by definition, a $J$-invariant 2-form. Let us also note here the important property, specific to dimension 4, that the symmetric part of the star-Ricci tensor ${\rm Ric}^{* sym}$ and the $J$-invariant part of the Ricci tensor ${\rm Ric}'$ have the same trace-free part. In other words (see \cite{TV}),
$$ {\rm Ric}^{* sym} - {\rm Ric}' = \frac{s^* - s}{4} \, g$$
where $ s = tr({\rm Ric})$, $s^* = tr({\rm Ric}^*) $ are the scalar, respectively, the star-scalar curvatures. In terms of Ricci forms, the above relation is equivalent with
$$ (\rho^*)'_0  = \rho_0 \; . $$

\vspace{0.2cm}

\noindent
Further useful is the Weitzenb\"ock formula for the fundamental form $\omega$
\begin{equation} \label{Weitz-om}
(d \delta + \delta d) \omega  = \nabla^* \nabla \omega + \frac{s}{3} \omega - 2W^+(\omega) \; .
\end{equation}
The $\omega$-component of the above yields a known relation between the scalar curvatures, $N$, and $\theta$, for an arbitrary almost Hermitian 4-manifold (see, e.g. \cite{Sek87})
\begin{equation} \label{kappa-s}
   s^* - s = \frac{2}{3}(\kappa - s) =  \frac{1}{4} |N|^2 - |\theta|^2 - 2 \delta \theta \; .
\end{equation}
We end this subsection with a series of equivalent local characterizations of Gray's first condition $(G_1)$ and some remarks.

\begin{prop} \label{ah1-locchar}
Let $(M^4, g, J, \omega)$ be an almost Hermitian 4-manifold. Then the following are equivalent:

(i) The manifold is in the class $\mathcal{AH}_1$;

(ii) ${\rm Ric}$ is $J$-invariant, $W^+_2 = 0$, $W^+_3 = 0$, and $\kappa - s = 0$;

(iii) For any vectors $X,Y$, $(\nabla^2_{X,Y} - \nabla^2_{Y,X}) \omega = 0$;

(iv) 
If $a, b, c$ are the 1-forms in relations (\ref{abc}) with respect to a gauge $\phi$, then 
$$da = c \wedge b \; , \; \; db = - c\wedge a \; ;$$

(v) For any vectors $X, Y$, the following identity  holds:
\begin{eqnarray} \label{invAH1-v2}
    & & \; \; \; \; \;   \frac{1}{2} \Big[ \theta(X) \big( Y \wedge J \theta + JY \wedge \theta \big)  - \theta(Y) \big( X \wedge J \theta + JX \wedge \theta \big) \Big] \\ \nonumber
    & &  + \frac{1}{2} |\theta|^2 \big( X \wedge J Y + JX \wedge Y \big) + \frac{1}{2} \Big( Y\wedge N_{JX}(\theta) - X\wedge N_{JY}(\theta) + JN(X,Y) \wedge \theta \Big)
    \\ \nonumber
      & & + Y \wedge J(\nabla_X \theta) - X \wedge J(\nabla_Y \theta) + JY \wedge (\nabla_X \theta) - JX \wedge (\nabla_Y \theta) - d^{\nabla}_{X,Y}(JN)  = 0\; . 
\end{eqnarray} 
\end{prop} 
\begin{proof} The equivalence (i) $\Leftrightarrow$ (ii) is well-known, see \cite{TV}; it follows from the $U(2)$-decomposition of the curvature (\ref{u(2)}) and the fact that Gray's first curvature condition is equivalent with $R|_{\leftr \La^{0,2}M \rightr} \equiv 0$. The equivalence (i) $\Leftrightarrow$ (iii)  holds in all dimensions and is due to the Ricci identity
\begin{equation*} 
( \nabla^2_{X,Y} -  \nabla^2_{Y,X} )\omega = -R_{X,Y}(J\cdot, \cdot) -R_{X,Y}(\cdot, J\cdot) \; .
\end{equation*}
Further, (iv) is just the relation (iii) written, in dimension 4, with respect with a gauge $\phi$. Indeed,  from (\ref{abc})
$$\na ^2|_{\Lambda^2 M}\omega =(da - c\wedge b)\otimes \phi + (db + c\wedge a)\otimes J\phi, $$ 
so, the equivalence (iii) $\Leftrightarrow$ (iv) is clear. Finally, (v) is just the relation (iii) expanded by (\ref{nablaom-dim4}), in an invariant form. The computation is a bit longer, but straightforward, so we simply indicate the main steps.
Taking one more derivative of (\ref{nablaom-dim4}), one has
\begin{equation*} 
\nabla^2_{X ,Y} \omega = \frac{1}{2} \Big( (\nabla_X J )Y \wedge \theta + Y \wedge \nabla_X(J \theta) + JY \wedge \nabla_X \theta - \nabla_X (JN)_Y  \Big) \; .
\end{equation*}
Skew-symmetrizing and using Gray's first curvature condition one eventually gets 
\begin{eqnarray*} 
    & & \big( (d \omega)(X,Y, \cdot) - (\nabla_{\cdot} \omega)(X,Y) \big) \wedge \theta 
      + Y \wedge (\nabla_X J)(\theta) - X \wedge (\nabla_Y J)(\theta) \\ \nonumber
      & & + Y \wedge J(\nabla_X \theta) - X \wedge J(\nabla_Y \theta) + JY \wedge (\nabla_X \theta) - JX \wedge (\nabla_Y \theta) - d^{\nabla}_{X,Y}(JN)  = 0\; .
\end{eqnarray*} 
Some further work using (\ref{nablaom-dim4}) and (\ref{theta-def}) on the terms in the first line eventually yields relation (\ref{invAH1-v2}).
\end{proof}

\begin{remark} \label{basic-exp}
Based on (ii), one easily obtains a family of examples of non-K\"ahler 4-dimensional $\mathcal{AH}_1$ manifolds. Assume that $(M^4, g)$ is a Ricci flat ASD 4-manifold and let $J$ be an arbitrary almost complex structure compatible with the given metric. Then $(M^4, g, J)$ is of class $\mathcal{AH}_1$, as the conditions in (ii) are trivially satisfied, but it is generally not K\"ahler, as most $J$'s compatible with the given metric are not even integrable. We will call these non-K\"ahler $\mathcal{AH}_1$ 4-manifolds arising from this remark {\bf basic $\mathcal{AH}_1$ examples}. In the compact case, applying Hitchin's classification of Einstein, half-conformally flat compact 4-manifolds with non-negative scalar curvature
(see \cite{besse}, Chapter 13), Ricci-flat ASD metrics exist only on the 4-torus $\mathbb{T}^4$ or some quotients of $\mathbb{T}^4$ (hyper-elliptic surfaces), or on $K3$-surfaces or finite quotients of $K3$-surfaces (Enriques surfaces). Thus, compact basic $\mathcal{AH}_1$ examples occur on these 4-manifolds equipped with a Ricci-flat ASD metric $g$ and an arbitrary $g$-compatible almost complex structure $J$.
\end{remark}

\noindent 
Regarding this remark, let us add that we do not know any non-K\"ahler, non-basic example of an $\mathcal{AH}_1$ 4-manifold, and we conjecture that, at least in the compact case, they do not exist (see Conjecture \ref{conjecture} in Section 5). 

\begin{remark} \label{AK1orH1=K}
    Let us also note here that the relation (\ref{kappa-s}) combined with characterization (ii) from the above proposition implies that the equality $\mathcal{AK}_1 = \mathcal{K} $ holds locally (actually, in all dimensions), while the equality $\mathcal{H}_1 = \mathcal{K}$ holds in the compact case (in dimension 4). 
\end{remark}

\subsection{The canonical Chern form and the form $\Phi$} 
For any almost Hermitian manifold, a 2-form representative of the first Chern class $2\pi c_1$  is given by
\begin{equation} \label{Chernform} 
\gamma(X, Y) = \rho^*(X, Y) + \Phi(X, Y) \; , 
\end{equation}
where $\gamma$ is the Ricci form of the first canonical Hermitian connection (see e.g. \cite{gauduchon1}), 
$$\na^0_X Y = \na_X Y - \frac{1}{2}J(\na_X J)(Y) \; ,$$ 
$\rho^* = R(\omega)$ is the star-Ricci form and the 2-form $\Phi$ is given by
\begin{equation*} \label{Phi} 
\Phi(X,Y) = \frac{1}{4} \langle J(\nabla_X J), (\nabla_Y J) \rangle = \frac{1}{2} \langle J(\nabla_X \omega), (\nabla_Y \omega) \rangle \; .
\end{equation*}
Note that we could have used any connection $\nabla^t$ of the family of natural Hermitian connections introduced by Gauduchon \cite{gauduchon1} for an almost Hermitian manifold. The Ricci form $\gamma^t$ of $\nabla^t$ is also a representative of the first Chern class $2\pi c_1$, and the first term in its expression is also the star-Ricci form. We prefer the first canonical connection (corresponding to $t=0$) just because the expression of its Ricci form $\gamma^0 = \gamma$ fits in best with the local 1-forms $a,b,c$ defined in (\ref{abc}). We actually have, a fact explained a few lines below,
\begin{equation} \label{gamma=-dc}
 \gamma = - dc   \; .
\end{equation}
Using (\ref{nablaom-dim4}), a direct calculation shows that the form $\Phi$ has the following expression for a 4-dimensional almost Hermitian manifold:
\begin{equation} \label{Phi4d}
\Phi(X,Y) = - \frac{1}{8} \langle N_{JX}, N_Y \rangle + \frac{1}{4} ( |\theta|^2 \omega - \theta \wedge J\theta)(X,Y) - \frac{1}{4} N_{J\theta^{\sharp}} (X,Y) \; .
\end{equation}
Note that 
\begin{equation} \label{Phi-comp}
   \langle\Phi, \omega\rangle = -\frac{1}{16} |N|^2 +\frac{1}{4} |\theta|^2  \; , \; \; \Phi'' = - \frac{1}{4} N_{J\theta^{\sharp}} \; .
\end{equation}  
Thus, the $J$-anti-invariant part of $\Phi$ vanishes if and only if
${\rm Span}(\theta, J\theta)$ is orthogonal to the image of the Nijenhuis tensor $N$. In terms of the 1-forms defined by (\ref{abc}) and (\ref{NphiJphi}), we get the following local expressions of $\Phi$ 
\begin{equation} \label{Phi4dloc}
\Phi = a \wedge b = - \frac{1}{4} n \wedge Jn + \frac{1}{4} \phi(\theta) \wedge J\phi(\theta)- \frac{1}{4} \big( n \wedge \phi(\theta) - Jn \wedge J \phi(\theta) \big) \; .
\end{equation}
This yields the remarkable fact that, in dimension 4, $\Phi \wedge \Phi = 0$.

Alternatively, as in \cite{AAD-AGAG} (see bottom of page 157), consider the Ricci relation applied to the gauge $\phi$
\begin{equation} \label{Ric-id-phi}
( \nabla^2_{X,Y} -  \nabla^2_{Y,X} )\phi = -R_{X,Y}(\phi\cdot, \cdot) -R_{X,Y}(\cdot, \phi\cdot) \; .
\end{equation}
Taking one more derivative in (\ref{abc}), the left side of (\ref{Ric-id-phi}) is given by 
$$\na ^2|_{\Lambda^2 M}\phi = - (da + b\wedge c)\otimes \omega + (dc  + a\wedge b)\otimes J\phi \; . $$
Taking the $J\phi$-component of this and also using (\ref{Ric-id-phi}), we get 
$$ dc + a \wedge b = -R(\omega) \; ,$$
which can be rewritten as
$$ dc = - R(\omega) - a\wedge b = - \gamma \; .$$
This justifies the claim made in (\ref{gamma=-dc}).

\vspace{0.2cm} 

In what follows, an assumption that the 2-form $\Phi$ is an exact form will be of relatively high importance. On the other hand, a computation of $dJ\theta$ using the definition of differential and relations (\ref{nablaom-dim4}), (\ref{theta-def}) yields (see also \cite{DS}, Proposition 2.3)
\begin{equation} \label{dJtheta}
dJ\theta = \nabla_{\theta^{\sharp}} \omega - \iota_{\theta^{\sharp}} d\omega  - \alpha = \frac{1}{2} N_{J\theta^{\sharp}} + \theta \wedge J\theta - |\theta|^2 \omega - \alpha \;,
\end{equation}
where $\alpha$ denotes the 2-form defined by
\begin{equation*} \label{alpha-def}
   \alpha(A, B) =  (\nabla_{A} \theta)(JB) - (\nabla_{B} \theta)(JA) \; .
\end{equation*}
Note that
$$ \langle \alpha, \omega \rangle =  \delta \theta \; , \; \alpha ' = -2\big( (\nabla \theta)^{sym} \big)' \circ J \; , \; \alpha'' = - J(d\theta)'' \; .$$
Thus, certain terms are similar in $\Phi$ and $dJ\theta$. Combining (\ref{dJtheta}) and (\ref{Phi4d}), we get 
$$ - 4 \Phi - dJ \theta = \frac{1}{2} \langle N_{J\cdot}, N_{\cdot} \rangle + \frac{1}{2} N_{J\theta^{\sharp}} + \alpha \; ,$$
where $\alpha$ is the 2-form defined above. Therefore, we have the following result:
\begin{prop} \label{Phi=dJth}
 On any almost Hermitian 4-manifold $(M^4, g, J, \omega)$, the equality $-4\Phi = dJ\theta$ holds if and only if 
    \begin{equation} \label{Phi=dJtheta}
        (\nabla_X \theta)(JY) - (\nabla_Y \theta)(JX) = -\frac{1}{2} \langle N_{JX}, N_{Y} \rangle - \frac{1}{2} N_{J\theta^{\sharp}}(X,Y) \; .
    \end{equation}   
Moreover, if the manifold is compact, the equality $-4\Phi = dJ\theta$ is equivalent with the manifold being Vaisman (i.e. $J$ is integrable and $\nabla \theta = 0$).    
\end{prop}
\begin{proof}
    Only the statement in the compact case needs proof, as the local equivalence holds by the formula found for $-4\Phi - dJ\theta$ just above the statement. The $\omega$-component of (\ref{Phi=dJtheta}) yields
$$ \delta \theta = -\frac{1}{4} |N|^2 \; ,$$    
so, in the compact case, by integration, we get that $N\equiv 0$, therefore the almost complex structure must be integrable. Using this, the $\leftr \La^{0,2}M \rightr$-component of (\ref{Phi=dJtheta}) becomes $J(d\theta)''= 0$. As $\langle d\theta , \omega \rangle = 0$ always holds, it follows that $d\theta$ is an anti-self-dual form. In the compact case this immediately implies that $d \theta =0$. Therefore, thus far, we have proved that the manifold is locally conformally K\"ahler (lcK). Note that relation (\ref{Phi=dJtheta}) becomes equivalent with the fact that the lcK manifold is pluricanonical (see \cite{Kok}). Finally, we use a result of Moroianu-Moroianu \cite{MM} which establishes that compact lcK pluricanonical manifolds (of any dimension) are necessarily Vaisman.
\end{proof}
    
\section{A global characterization of compact $\mathcal{AH}_1$ 4-manifolds}

In this section, we use an identity established by Sekigawa \cite{Sek87} for compact almost Hermitian 4-manifolds to obtain a global characterization of compact $\mathcal{AH}_1$ 4-manifolds from seemingly weaker conditions. For completeness, we review Sekigawa's proof, using our conventions and notation, and we set the identity in a favorable form for our purpose.


Following \cite{Sek87}, we compute the Chern number $c_1^2(M)$ in two different ways. The first way is directly, using the expression $\gamma$ above:
\begin{eqnarray}  \nonumber
c_1^2(M) &=& \frac{1}{4 \pi^2} \int_M \gamma \wedge \gamma =  \frac{1}{4 \pi^2} \int_M ( \rho^* \wedge \rho^* + 2 \rho^* \wedge \Phi ) = \\ \nonumber
&=&  \frac{1}{4 \pi^2} \int_M \Big( |{\rho^*}''|^2 + \frac{(s^*)^2}{8} - |{\rho_0^*} '|^2 \Big) dV_g 
+ \frac{1}{4 \pi^2} \int_M 2 \rho^* \wedge \Phi \; .
\end{eqnarray}
Using that in dimension 4 ${\rho_0^*}' = \rho_0$, we get 
\begin{equation} \label{c1sqv1}
c_1^2(M) = \frac{1}{4 \pi^2} \int_M \Big( |{\rho^*}''|^2 + \frac{(s^*)^2}{8} - |\rho_0|^2 \Big) dV_g
+ \frac{1}{4 \pi^2} \int_M 2 \rho^* \wedge \Phi \; .
\end{equation}
The second way, we compute $c_1^2(M)$ via Chern-Weil formulae
$$c_1^2(M) = 2\chi(M) + 3 \sigma(M) = \frac{1}{4 \pi^2} \int_M \big( \frac{s^2}{24} - \frac{1}{2} |{\rm Ric}_0 |^2 + 2 |W^+|^2 \big) dV_g \; .$$
Decomposing ${\rm Ric}_0$ into its $J$-invariant part ${\rm Ric}_0^{'}$ and its $J$-anti-invariant part ${\rm Ric}_0^{''}$, note that 
$$ |{\rm Ric}_0|^2 = | {\rm Ric}_0^{'} |^2 + |{\rm Ric}_0^{''}|^2 = 2 |\rho_0 |^2 + |{\rm Ric}_0^{''}|^2 \; .$$
Similarly, using the $U(2)$-decomposition of $W^+$ from the previous section, we have
$$|W^+|^2 = |W^+_1|^2 + |W^+_2|^2 + |W^+_3|^2 = \frac{\kappa^2}{24} + |{\rho^*}''|^2 + |W^+_3|^2 \; .$$
Using these, one obtains the following expression:
\begin{equation} \label{c1sqv2}
c_1^2(M) = \frac{1}{4 \pi^2} \int_M \Big( \frac{s^2}{24} - |\rho_0 |^2 - \frac{1}{2} |{\rm Ric}_0^{''} |^2 + \frac{\kappa^2}{12} + 2|{\rho^*}''|^2 + 2|W^+_3|^2 \Big) dV_g 
\end{equation}

\noindent
Subtracting (\ref{c1sqv1}) from (\ref{c1sqv2}), one gets Sekigawa's integral identity for a closed almost Hermitian 4-manifold. We write this identity in a form that best suits our interest here.
\begin{prop} \label{if-prop} (\cite{Sek87})
On any compact almost Hermitian 4-manifold $(M^4, g, J, \omega)$ we have
\begin{equation} \label{if}
0 = \int_M \Big[ \frac{(s^* - s)^2}{16} + |{\rho^*}''|^2 + 2 |W_3^+|^2 - \frac{1}{2} |{\rm Ric}_0''|^2 \Big] dV_g - 
2 \int_M \rho^* \wedge \Phi \; .
\end{equation}
\end{prop}
\noindent
The proof is already given above, noting that, via the definition of $\kappa$
$$ \frac{s^2}{24} + \frac{\kappa^2}{12} - \frac{(s^*)^2}{8} = \frac{(s^* - s)^2}{16} \; .$$
It is also worth observing that the second integral on the right side of (\ref{if}) is usually expanded as
$$\int_M \rho^* \wedge \Phi =  \int_M \langle ({\rho^*}'' - \rho_0 + \frac{s^*}{4} \omega), \Phi \rangle dV_g\; ,$$
where one uses that, in dimension 4, $\rho^*_0 = \rho_0$. Our preference to leave the term
$ \displaystyle{ \int_M \rho^* \wedge \Phi }$ in this form is motivated by the Remarks \ref{suff-conds} and \ref{equiv-forms} below.   

\vspace{0.2cm}

\noindent We use Proposition \ref{if-prop} to prove
\begin{prop} \label{ah1}
Let $(M^4, g, J, \omega)$ be a compact almost Hermitian 4-manifold. Then the following are equivalent:

(i) The Ricci tensor is $J$-invariant and $\displaystyle{\int_M \rho^* \wedge \Phi = 0}$;
 
(ii) The manifold is in the class $\mathcal{AH}_1$, that is $(M^4, g, J, \omega)$ satisfies Gray's first curvature condition $(G_1)$.
\end{prop}

\begin{proof}
For (i) $ \Rightarrow $ (ii), relation (\ref{if}) and the assumptions imply 
$$ 0 = \int_M \Big(  \frac{(s^* - s)^2}{16} + |{\rho^*}''|^2 + 2 |W_3^+|^2 \Big) dV_g \; .$$
All three terms must vanish and the conclusion follows from Proposition \ref{ah1-locchar}, (ii).
The implication (ii) $ \Rightarrow $ (i) again follows from (\ref{if}).
%
%
\end{proof}

\begin{remark} \label{suff-conds}
Note that the equality $ \displaystyle{ \int_M \rho^* \wedge \Phi  = 0}$ follows from the assumption that $\rho^*$ or $\Phi$ is an exact form. Indeed, in such a case, the other is a closed form (because their sum is the closed canonical Chern form $\gamma$), so $\rho^* \wedge \Phi$ is an exact 4-form, whose integral is $0$ by Stokes' Theorem. As we will see in Section 5, on a $\mathcal{AH}_1$ 4-manifold, it follows that $\rho^* \wedge \Phi = 0$ actually holds point-wise.
\end{remark}

\begin{remark} \label{equiv-forms}
A different formulation of Proposition \ref{ah1} is: ``Let $(M^4, g, J, \omega)$ be a compact almost Hermitian 4-manifold with a $J$-invariant Ricci tensor. Then
$$\int_M \rho^* \wedge \rho^* \leq \frac{1}{4\pi^2} c_1^2(M) \; ,$$ 
with equality if and only if the manifold is of class $\mathcal{AH}_1$.''

\noindent
Indeed, by the definition of $\gamma$ and the fact that $\Phi \wedge \Phi = 0$, the above inequality is equivalent with 
$$ \int_M \rho^* \wedge \Phi  \geq 0 \; ,$$
and the equality case follows just as described in the proof of Proposition \ref{ah1}.
\end{remark}

\noindent
Finally, note that both (i) and (ii) in Proposition \ref{ah1} include the case that the manifold is K\"ahler, as well as the case of the basic $\mathcal{AH}_1$ examples from Remark \ref{basic-exp}. As already mentioned, we conjecture that there are no other non-K\"ahler examples in the compact case. The last two results of this section are immediate consequences of Proposition \ref{ah1} and can be seen as partial support for this conjecture.

\begin{cor} \label{herm-case}
Let $(M^4, g, J, \omega)$ be a compact Hermitian 4-manifold with $J$-invariant Ricci tensor and that satisfies $\displaystyle{\int_M \rho^* \wedge \Phi = 0}$. Then $(M^4, g, J, \omega)$ must be a K\"ahler surface.
\end{cor}
\begin{proof} By Proposition \ref{ah1}, we have that $(M^4, g, J, \omega)$ satisfies Gray's 1st curvature condition, so, using the Hermitian condition, by (\ref{kappa-s}) we have
$$ 0 = s^* - s = -|\theta|^2 -2\delta \theta  \; .$$
Integrating this implies $\theta = 0$.
\end{proof} 

\begin{cor}
Let $(M^4, g, J, \omega)$ be a compact Vaisman 4-manifold with a $J$-invariant Ricci tensor. Then $(M^4, g, J, \omega)$ must be a K\"ahler surface.
\end{cor}
\begin{proof}
For a Vaisman 4-manifold, $\Phi = - \frac{1}{4} dJ\theta$, as seen in Proposition \ref{Phi=dJth}, so the condition
$\displaystyle{\int_M \rho^* \wedge \Phi = 0}$ holds, as observed in Remark \ref{suff-conds}.
\end{proof}

\section{The differential Bianchi identity for almost Hermitian 4-manifolds} 

This is the most technical part of our paper, but it is essential for the steps we take in the next section toward a classification of 4-dimensional $\mathcal{AH}_1$ 4-manifolds. We also hope that the lemmas contained here will be helpful for other projects involving 4-dimensional almost Hermitian geometry, so we try to obtain the results in full generality. We actually get a great help from \cite{AAD}, as all the results here have been obtained for almost K\"ahler 4-manifolds in Subsection 2.4 of that paper. The computation techniques are generally the same as in the almost K\"ahler case, we just have to keep track of the extra terms in $\theta$ that will appear. The main goal is to express the information contained in the two halves (self-dual and anti-self-dual) of the differential Bianchi identity
\begin{equation*}\label{bianchipm}
 \delta W ^+  = C ^+, \ \  \delta W ^-  = C ^- ,
\end{equation*}
by further decomposing the above in terms of various $U(2)$-components of the curvature. We want this information expressed in terms of forms (rather than tensors) with the purpose of eventually applying Hodge theory, under suitable assumptions. Above, $C$ is the {\it Cotton-York tensor}, a section of $\Lambda^2 M \otimes \Lambda^1 M$ defined by
$$C(X, Y, Z) = - (d^{\na} h)(X, Y, Z) = - (\nabla _X h )(Y, Z) + (\nabla_Y h)(X, Z),$$ 
where $h = \frac{1}{2}  {\rm Ric}_0 + \frac{s}{24} g$ denotes the {\it normalized Ricci tensor}, seen as a $\Lambda^1 M$-valued 1-form via the metric, and $d^{\na}$ is the
Riemannian differential acting on $\Lambda^1M$-valued forms
$$ d^{\na} : \Lambda^1 M \otimes \Lambda^1 M \rightarrow \Lambda^2 M \otimes \Lambda^1 M \; .$$
We start by re-arranging  the well-known formula
$$\delta ({\rm Ric}_0 - \frac{s}{4} g) = 0 \; ,$$
which is obtained from a contraction of the full Bianchi identity $\delta W = C$. 
\begin{lemma} \label{RicformBianchi}
For any 4-dimensional almost Hermitian manifold $(M^4, g, J, \omega)$ the following identities hold:
\begin{equation} \label{deltarho0}
\delta \left( \rho_0 - \frac{s}{4}  \omega \right)  =  \frac{s}{4} J\theta + \iota_{\theta^\sharp} \rho_0 - J\delta ({\rm Ric}_0'') \; ;
\end{equation}
\begin{equation} \label{drho}
    d \rho = \frac{s}{4} \theta \wedge \omega - \theta \wedge \rho_0  - \star (J\delta ({\rm Ric}_0'') )          \; ;
\end{equation}
\begin{equation} \label{drho0}
  d \rho_0 = - \frac{1}{4} ds \wedge \omega - \theta \wedge \rho_0   - \star (J\delta ({\rm Ric}_0'') )          \; .
\end{equation}
\end{lemma}


\begin{proof}
For any $J$-invariant symmetric tensor $b \in S^{1,1}_{\mathbb{R}} M$ on an almost Hermitian manifold (of any dimension), one has 
\begin{equation} \label{deltabJinv}
\delta(b) = ( (\iota_{\cdot} d \omega)' , b \circ J )_g - J \delta(b \circ J)  \; , 
\end{equation}
which could be checked by a direct computation, using (\ref{nablaom-dim4}).
Next, apply (\ref{deltabJinv}) to the second term of 
$$\delta ({\rm Ric}_0'') + \delta ({\rm Ric}_0' - \frac{s}{4} g) = 0 \; ,$$ and use that in dimension 4
$$\iota_X d \omega = \theta(X) \omega - \theta \wedge JX^\flat \; .$$
The relation (\ref{deltarho0}) then follows by simply rearranging the terms. The relation (\ref{drho}) follows by taking the Hodge star of both sides of (\ref{deltarho0}), and noting that 
$$ \star (\iota_{\theta^\sharp} \rho_0) = - \star \iota_{\theta^\sharp} \star \rho_0 =  \theta \wedge \rho_0 \; .$$
Finally, relation (\ref{drho0}) follows by some rearranging of (\ref{drho}).
\end{proof}

Next, we obtain an alternative expression for the Cotton-York tensor $C$.
We follow the notation from \cite{AAD}, that is, for a given vector field $Z$, we denote by $C_Z$  the section of
$\La^2 M$,  defined  by  $C_Z (X,Y): = C(X,Y,Z)$, and similarly we
define $C^+_Z$ and $C^-_Z$. Denote also by $A_Z$ the
$\La^2M$-valued 1-form given by
$$A_Z  = (d^{\nabla} {\rm Ric}^{''})_{Z} - \iota_{JZ}\star(J \delta {\rm Ric}'') \; ,$$
and let $A^{\pm}_Z$ be the self-dual and  anti-self-dual components of
$A_Z$. The following is the correspondent of Lemma 3 from \cite{AAD}, adapted to almost Hermitian 4-manifolds.
%

\begin{lemma} \label{lemma-CY-AH} Let $(M, g, J,\omega)$  be an almost Hermitian 4-manifold. Then, for any vector field $Z$, the
Cotton-York tensor $C_Z$ is given by
\begin{eqnarray} \label{cotton-york}
2C_Z &= &\nabla_{JZ} \, \rho_0 - (J\theta)(Z) \rho_0 -\frac{1}{4}(Jds)(Z) \, \omega +(J{\rm Ric}_0'(\theta))(Z) \, \omega \\ \nonumber 
& & +(\theta \wedge {\rm Ric}_0'(Z))'' + \frac{1}{2} N_{J\rho_0(Z)^\sharp}
 + \frac{1}{6} ds \wedge Z^{\flat} - A_Z \; .
\end{eqnarray}
\end{lemma}
\begin{proof}
Up to a point, the proof goes just in \cite{AAD}. The Cotton-York tensor is rewritten as
\begin{equation}\label{basicpart}
C_Z = - \frac{1}{2} (d^{\na}{\rm Ric}^{''})_Z - \frac{1}{2} (d^{\na}
{\rm Ric}^{'})_Z + \frac{1}{12} ds\wedge Z^{\flat} \; ,
\end{equation}
and the main effort will be in computing the middle term.
As $\rho(\cdot, \cdot) = {\rm Ric}^{'} (J\cdot, \cdot)$, we have
\begin{equation} \label{dricinv} \begin{split}
 (d^{\na} {\rm Ric}^{'})(X, Y, Z) &= (\nabla _X {\rm Ric}^{'})(Y, Z) -  (\nabla
_Y {\rm Ric}^{'}) (X, Z) \\ &= (d^{\na} \rho)(X,Y,JZ) + \Big(
\rho(Y, (\na_X J)Z) - \rho(X, (\na_Y J)Z) \Big) .
\end{split} \end{equation}
For the term $d^{\na} \rho$ we have
\begin{equation} \nonumber \begin{split}
(d^{\na} \rho)(X,Y,JZ) &= (\na_X \rho)(Y,JZ) - (\na_Y \rho)(X,JZ)
\\ &= - (\na_{JZ} \rho)(X,Y) + (d\rho)(X,Y,JZ) \\
&= - (\na_{JZ} \rho)(X,Y) + \iota_{JZ}(d\rho)(X,Y) \\
&= - (\na_{JZ} (\rho_0 +\frac{s}{4} \omega)) (X, Y) +  \iota_{JZ}(d (\rho_0 + \frac{s}{4} \omega))(X,Y) \;.
\end{split} \end{equation}
After some straightforward computations, we get
\begin{equation} \nonumber \begin{split}
(d^{\na} \rho)_{JZ} &= - \nabla_{JZ} \rho_0 - \frac{s}{4} (\nabla_{JZ} \omega) + \iota_{JZ} (d \rho_0) \\
& + \frac{1}{4} ds \wedge Z^{\flat} - \frac{s}{4} (J\theta)(Z) \omega + \frac{s}{4} (\theta \wedge Z^{\flat}) \; .
\end{split} \end{equation}
For the last term of the first line, we use relation (\ref{drho0}) to obtain
$$\iota_{JZ} (d \rho_0) = \frac{1}{4} (Jds)(Z) \omega -  \frac{1}{4} ds \wedge Z^{\flat} + (J\theta)(Z) \rho_0 - \theta \wedge {\rm Ric}_0'(Z) - \iota_{JZ} \star(J \delta {\rm Ric}'') \; . $$
We replace this in the above equality to eventually get
\begin{equation} \label{dricinvterm1} \begin{split}
(d^{\na} \rho)_{JZ} &= - \nabla_{JZ} \rho_0 + \frac{s}{4} (\theta \wedge Z^{\flat})^{-} + (J\theta)(Z) \rho_0 +\frac{1}{4} (Jds)(Z) \omega \\
& - \frac{s}{8} (J\theta)(Z) \omega + \frac{s}{8} N_Z - \theta \wedge {\rm Ric}_0'(Z) - \iota_{JZ} \star (J\delta {\rm Ric}'')\; .
\end{split} \end{equation}

\noindent
For the last term of (\ref{dricinv}), as in \cite{AAD}, we observe that $\na_X J$ corresponds to a self-dual form $\na_X \omega$, and this commutes with the corresponding endomorphism associated with the anti-self-dual form $\rho_0$. Thus, we obtain
\begin{equation} \label{dricinvterm2} 
 \rho(Y, (\na_X J)(Z)) - \rho(X, (\na_Y J)(Z) )  = 
\end{equation} 
$$= \frac{s}{4}\Big((\na_Y \omega)(X, JZ) -(\na_X \omega)(Y, JZ) \Big
) + (\na_X \omega)(Y, \rho_0(Z)) - (\na_Y \omega) (X,
\rho_0(Z))  $$ 
$$ = \frac{s}{4} \Big( (\na_{JZ} \omega) (X,Y) - (\iota_{JZ} d \omega) (X,Y) \Big) - \Big( (\na_{{\rho_0} (Z)} \omega)(X,Y) - (\iota_{{\rho_0} (Z)} d \omega)(X, Y) \Big) \; .$$
A short computation using (\ref{nablaom-dim4}) shows that for any tangent vector $V$
$$ \nabla_{V} \omega - \iota_{V} d \omega = -\frac{1}{2} \theta(V) \omega + \frac{1}{2} N_{JV} + (V \wedge J\theta)^{-} \; .$$
Applying this successively in (\ref{dricinvterm2}) for $V = JZ$ and $V= {\rho_0} (Z)^\sharp$, we eventually get 
\begin{equation} \label{dricinvterm2fin} \begin{split}
\rho &(Y, (\na_X J)(Z)) - \rho(X, (\na_Y J)(Z) )  = \\ & 
= \frac{s}{8} (J\theta)(Z) \omega -  \frac{s}{8} N_Z + \frac{1}{2} \theta(\rho_0(Z)^\sharp) \omega \\ & - \frac{1}{2} N_{J\rho_0(Z)^\sharp} - \frac{s}{4} (\theta \wedge Z^{\flat})^{-} + (\theta \wedge {\rm Ric}_0'(Z))^{-}   \; .
\end{split} \end{equation}
Substituting (\ref{dricinvterm1}) and (\ref{dricinvterm2fin}) back in (\ref{dricinv}), after a bit more work, we get
\begin{equation} \label{dricinv+} \begin{split}
(d^{\na} {\rm Ric}^{'})_Z &= - \na_{JZ} \rho_0 + (J\theta)(Z) \rho_0 + \frac{1}{4}(J ds)(Z) \omega - (J{\rm Ric}_0'(\theta))(Z) \omega \\ & \ \ \ - \frac{1}{2} N_{J\rho_0(Z)^\sharp} - \big(\theta \wedge {\rm Ric}_0'(Z) \big)'' - \iota_{JZ} \star (J\delta {\rm Ric}'')\; .
\end{split} \end{equation}
Finally, relation (\ref{cotton-york}) follows from (\ref{basicpart}), (\ref{dricinv+}) and the definition of $A_Z$. 
\end{proof} 

\noindent
The next lemma rewrites the differential Bianchi
identities $\delta W ^{+}  = C ^{+}, \ \delta W ^{-}  = C ^{-} $ for an arbitrary almost Hermitian 4-manifold, so it is an extension of Lemma 3 from \cite{AAD} and concludes this section.

\begin{lemma} \label{Bianchi-lem} Let $(M^4, g, J, \omega)$ be an
almost Hermitian 4-manifold. Then the following identities hold: 
\begin{eqnarray}\label{ahbianchi+} \nonumber
 0 &=& -\frac{1}{4} (Jd(\kappa - s))(Z)\,\omega + \frac{1}{6} (d(\kappa - s) \wedge Z^{\flat})^+ - \frac{\kappa}{4} (J\theta)(Z) \, \omega - (J{\rm Ric}_0'(\theta))(Z) \, \omega \\ 
& & +\frac{\kappa}{4} (\theta \wedge Z^{\flat} )'' - (\theta \wedge {\rm Ric}_0'(Z))'' + \frac{1}{2} N_{{\rm Ric}_0'(Z)^\sharp} - \frac{\kappa}{8} N_Z   \\ \nonumber
& & + (\delta {\rho^*}'')(Z) \, \omega + \nabla_{{\rho^*}''(Z)^\sharp} \omega - (J\theta)(Z) \, {\rho^*}''  + \nabla_{JZ} {\rho^*}'' + 2(\delta W_3^+)_Z + A^+_Z  \; ; 
\end{eqnarray}
\begin{eqnarray}
\label{ahbianchi-} 0 &=& \nabla_{JZ} \rho_0 -(J\theta)(Z) \, \rho_0 + \frac{1}{6} (d s \wedge Z^{\flat})^- - 2\delta W^-_{Z} - A^-_Z \; .
\end{eqnarray}
\end{lemma}
\begin{proof} The proof is just as in \cite{AAD}; relation (\ref{ahbianchi-}) follows from $\delta W^{-} = C ^{-}$ by taking the anti-self-dual component of
(\ref{cotton-york}). Similarly, for (\ref{ahbianchi+}) we take the
self-dual component of (\ref{cotton-york}); we also use the
following expressions for $\delta W_1^+$ and $\delta W_2^+$ obtained after some straightforward computations starting from (\ref{w^+1}) and (\ref{w^+2}).
\begin{equation*}
(\delta W_1^+)_Z = -\frac{1}{8} (Jd\kappa)(Z) \, \omega -\frac{\kappa}{8} (J\theta)(Z) \, \omega + \frac{\kappa}{8} (\theta \wedge Z^{\flat})'' - \frac{\kappa}{16} N_Z  + \frac{1}{12} (d\kappa \wedge Z^{\flat})^+ ,
\end{equation*}
\begin{equation*} 
(\delta W^+_2)_Z = \frac{1}{2} \nabla_{JZ} ({\rho^*}'')  +
\frac{1}{2} \nabla_{({\rho^*}'')(Z)} \omega  - \frac{1}{2} (J\theta)_Z \, \rho^{*}{''} + \frac{1}{2} (\delta
\rho^{*}{''})(Z) \, \omega \; . \hspace{1.8cm}
\end{equation*}

\end{proof}

\section{Towards a classification of $\mathcal{AH}_1$ 4-manifolds}

In this section, we take some steps towards a classification of almost Hermitian 4-manifolds satisfying the condition $(G_1)$. At least in the compact case, we believe that there is strong evidence for the following.


\begin{conj} \label{conjecture}
If $(M^4, g, J, \omega)$ is a compact 4-dimensional almost Hermitian manifold of class $\mathcal{AH}_1$, then either the manifold is K\"ahler, or it is one of the basic examples described in Remark \ref{basic-exp}. 
\end{conj}
\noindent
Theorem \ref{Eah1} from the Introduction is a partial answer supporting this conjecture, as are most of the results proved in this section. 



\vspace{0.2cm}
\noindent
We start by establishing some local consequences of the differential Bianchi identities combined with condition $(G_1)$ on an almost Hermitian 4-manifold.

\begin{prop} \label{ah1-facts}
Let $(M^4, g, J, \omega)$ be an almost Hermitian 4-manifold of class $\mathcal{AH}_1$. Then the following hold:

\vspace{0.2cm}
(a) ${\rm Ric}_0(\theta) = -\frac{s}{4} \theta$ and ${\rm Ric}_0(J\theta) = -\frac{s}{4} J\theta$;

\vspace{0.2cm}
(b) At all points where $s \neq 0$, ${\rm Span}(\theta, J\theta)$ and ${\rm Image} \, N $ are orthogonal. 

\vspace{0.2cm}

(c) ${\rm Ric}_0(X) = \frac{s}{4} X$, for any vector $ X \in {\rm Image} \, N$;

\vspace{0.2cm}

(d) At all points where $s = 0$, we have ${\rm Ric} = 0$ or both ($\theta = 0$ and $N=0$).

\vspace{0.2cm}

(e) We have $d \rho = 0$ and $d \Phi = 0$ everywhere on $M$. Moreover,  $\rho$ and $\Phi$ are collinear or one of them is zero. In particular, at points where $s\neq 0$, $\theta \neq 0$ and $N\neq 0$, 
$$\rho = \frac{s}{2} \; \omega|_{{\rm Image}\, N}  \; ,$$
$$ \Phi = \frac{1}{4} \left( |\theta|^2 - \frac{1}{4}|N|^2 \right) \omega|_{{\rm Image}\, N} = -\frac{1}{2} (\delta \theta)\; \omega|_{{\rm Image}\, N}\; .$$

\vspace{0.2cm}

(f) Relation $\rho^* \wedge \Phi = 0$ holds point-wise on $M^4$.
\end{prop}


\begin{proof}
For parts (a), (b), (c) we will use the self-dual Bianchi relation (\ref{ahbianchi+}), which, under the $\mathcal{AH}_1$ assumption becomes:
\begin{eqnarray} \label{bianchi+AH1}
 0 &=&  - \frac{s}{4} (J\theta)(Z) \, \omega - (J{\rm Ric}_0 (\theta))(Z) \, \omega \\ \nonumber
& & +\frac{s}{4} (\theta \wedge Z^{\flat} )'' - (\theta \wedge {\rm Ric}_0(Z))'' + \frac{1}{2} N_{{\rm Ric}_0 (Z)^\sharp} - \frac{s}{8} N_Z   \; .
\end{eqnarray}
The $\omega$ component of the relation above, immediately implies
${\rm Ric}_0 (J\theta) = - \frac{s}{4} J\theta$, hence also 
${\rm Ric}_0 (\theta) = - \frac{s}{4} \theta $. See also Lemma 4 in \cite{AD-QJM}, for part (a). Part (b) follows from the $J$-anti-invariant component of the Bianchi relation (\ref{bianchi+AH1}) by making $Z = \theta^{\sharp}$ and using part (a). Indeed, we obtain 
$$ - \frac{s}{8} N_{\theta^{\sharp}} = 0 \; ,$$ so claim (b) follows. At points where $\theta = 0$, part (c) follows directly from (\ref{bianchi+AH1}). At points where $\theta \neq 0$, part (c) follows from parts (a) and (b) and the fact that the trace-free Ricci tensor ${\rm Ric}_0$, being $J$-invariant, must have the double eigenvalue $\frac{s}{4}$, as it has the double eigenvalue $-\frac{s}{4}$ .

For part (d), suppose that there is a point $p$ where $s(p) = 0$, ${\rm Ric}_0(p) \neq 0 $ and $\theta(p) \neq 0$. Then the $J$-invariant 2-form $\rho_0$ has at each point at least a 2-dimensional kernel, as $\{ \theta(p), \; J\theta(p) \}$ are included in the kernel by part (a).
It follows that $\rho_0$ is decomposable at every point as $w \wedge Jw$, for some $w \in T^*M$. Thus, 
$$ 0 = \rho_0(p) \wedge \rho_0(p) = - |\rho_0 (p) |^2 \, \frac{\omega^2}{2} \; ,$$
where the second equal sign holds because $\rho_0$ is an anti-self-dual form. Thus $\rho_0 (p) = 0$, so ${\rm Ric}_0(p) = 0$, contradiction.
Similarly, a contradiction can be reached assuming the existence of a point $p$ such that $s(p) = 0$, ${\rm Ric}_0(p) \neq 0 $ and $N(p) \neq 0$, using part (c) in this case.

For the part (e), we first prove $d \rho = 0$.
Using part (a), we have $\iota_{\theta^\sharp} \rho_0 = -\frac{s}{4} J\theta$, so from Lemma \ref{RicformBianchi}, equation (\ref{deltarho0}), we get 
\begin{equation} \label{delrho0}
     \delta \left( \rho_0 - \frac{s}{4}  \omega \right)  = 0 \; . 
\end{equation} 
The Hodge star operator applied to this relation yields $d \rho = 0$. The claim $d \Phi = 0$ follows from (\ref{Chernform}), the fact that the first Chern form is closed, and the identity $\rho = \rho^*$, valid under the $\mathcal{AH}_1$ assumption. The expression for the Ricci form $\rho$
at points where $s\neq 0$, $\theta \neq 0$ and $N\neq 0$ follows from  parts (a) and (b) which imply that the trace-free Ricci tensor has the expression
$$ {\rm Ric}_0 = - \frac{s}{4} g|_{{\rm Span}(\theta, J\theta)} + \frac{s}{4} g|_{{\rm Image} \, N} \; .$$
The expression for $\Phi$ at points where $s\neq 0$, $\theta \neq 0$ and $N\neq 0$ follows from (\ref{Phi4d}) and the fact that $N_{J\theta} = 0$ by part (c). The claim that forms $\rho$ and $\Phi$ are collinear or one of them is $0$ at all points is now clear from the above expressions and part (d). Part (f) is also obvious now, noting that $\rho^* = \rho$ under the $\mathcal{AH}_1$ assumption.


  
\end{proof}

\noindent
As a consequence of the previous proposition, under the additional assumption that the metric is Einstein, we have a complete classification, even without compactness.

\begin{prop} \label{Einst-loc}
Let $(M^4, g, J, \omega)$ be an Einstein almost Hermitian 4-manifold of class $\mathcal{AH}_1$. Then either the scalar curvature is zero, in which case the metric is Ricci flat and ASD and $J$ is any metric compatible almost complex structure, or $(g, J, \omega)$ is a K\"ahler-Einstein structure.  
\end{prop}

\begin{proof}
If $s = 0$, then $\kappa = 0$, so $W^+ =0$ using Proposition \ref{ah1-locchar} (ii). As we have already observed in Remark \ref{basic-exp}, for a Ricci flat, ASD metric, any compatible almost complex structure $J$ gives rise to an $\mathcal{AH}_1$ structure. Suppose next that $s \neq 0$. As the metric is assumed Einstein, ${\rm Ric}_0 = 0$, so parts (a) and (c) of Proposition \ref{ah1-facts} imply that $\theta = 0$ and $N=0$. This can also be seen directly from the relation (\ref{bianchi+AH1}). Thus, the structure must be K\"ahler in this case.    
\end{proof}
The proof of Theorem \ref{Eah1} stated in the introduction now follows immediately.

\vspace{0.2cm}
\noindent
{\it Proof of Theorem \ref{Eah1}:} From the assumptions and Proposition \ref{ah1}, the manifold is Einstein and of class $\mathcal{AH}_1$. Now Proposition \ref{Einst-loc} applies, so the manifold is K\"ahler-Einstein if $s\neq 0$. If $s=0$, the metric is Ricci flat and ASD, and the conclusion follows from Hitchin's classification of compact Einstein ASD 4-manifolds with $s\geq 0$, as explained in Remark \ref{basic-exp}. $\Box$

\vspace{0.2cm}
\noindent Regarding the non-K\"ahler cases of Theorem \ref{Eah1} on $\mathbb{T}^4$ and its quotients, the metric is flat (see also the Lie algebra example in Section 6). However, note that in the case of hyperelliptic surfaces, only the metric and one K\"ahler structure descend to the quotient, not the entire hyperK\"ahler structure from $\mathbb{T}^4$. The same phenomenon holds for the Enriques surfaces obtained from quotients of $K3$-surfaces.

\vspace{0.2cm}
\noindent
For the next result, we will use compactness and an assumption on the scalar curvature.

\begin{prop} \label{ah1-Gaud}
    Let $(M^4, g, J, \omega)$ be a compact almost Hermitian 4-manifold of class $\mathcal{AH}_1$. Assume also that $s\neq 0$ everywhere on $M$. Then $\delta \theta = 0$. 
\end{prop}

\begin{proof}
   Assume that $\theta$ does not vanish identically, as otherwise there is nothing to prove. From parts (d) and (e) of Proposition \ref{ah1-facts}, we have
$$ s \, \Phi + (\delta \theta) \, \rho = 0 \; ,$$    
everywhere on $M$. Divide by $s$ and take the differential, using that $\Phi$ and $\rho$ are both closed, to obtain
%
%
%
%
$$ d \left( \frac{\delta \theta}{s} \right) \wedge \rho = 0\; . $$
Next, take the interior product $\iota_{\theta^{\sharp}}$ of this relation and note that $\iota_{\theta^{\sharp}} \rho  =0$ by part (a) of Proposition \ref{ah1-facts}
We get 
$$ \iota_{\theta^{\sharp}} d \left( \frac{\delta \theta}{s} \right) \,\rho = 0\; ,$$
and, as $\rho \neq 0$ (because of Proposition \ref{ah1-facts} (c) and the assumption $s \neq 0$), we get 
$$ \left( d \left( \frac{\delta \theta}{s} \right), \theta \right)_g = 0 $$
pointwise on $M$. Integrating this, one obtains
$$\int_M \frac{(\delta \theta)^2}{s} \; d\mu = 0 \; . $$
The conclusion follows, using again the assumption $s \neq 0$ on $M$.
\end{proof}

\noindent The next proposition is a computation on the basic examples from Remark \ref{basic-exp} which shows that at least some assumption on $s$ is needed in order for $\delta \theta $ to vanish on an $\mathcal{AH}_1$ 4-manifold.


\begin{prop} \label{basic-exp-comp}
Let $(M^4, g)$ be a Ricci-flat ASD 4-manifold and let us fix a local hyperK\"ahler structure $(g, J_1, J_2, J_3, \omega_1, \omega_2, \omega_3)$ (it is known that $g$ locally admits such structures).
 For any almost Hermitian structure $(g,J,\omega)$ compatible with the metric $g$, there are (locally defined) functions $f_1, f_2, f_3$ with $f_1^2 + f_2^2 + f_3^2 = 1$ so that $\omega = f_1 \, \omega_1 + f_2 \, \omega_2 + f_3 \,  \omega_3$. In terms of the functions $f_1, f_2, f_3$, some of the invariants of the almost Hermitian structure $(g,J,\omega)$ are given by:
\begin{equation} \label{theta-T4-exp}
\theta = J_3(f_2 df_1 - f_1 df_2) + J_1 (f_3 df_2 - f_2 df_3) +J_2(f_1df_3 -f_3 df_1) 
\end{equation}
\begin{eqnarray} \label{N-T4-exp}
N^J &=& (f_2 df_3 - f_3 df_2 +J(df_1)) \otimes \omega_1 + (f_3 df_1 - f_1 df_3 +J(df_2)) \otimes \omega_2 + \\ \nonumber & & \; \; \; \; \; \; \; \; \; \; \; + (f_1 df_2 - f_2 df_1 +J(df_3)) \otimes \omega_3
\end{eqnarray}
\begin{equation} \label{deltheta-T4-exp}
\delta \theta = -2\big( \, \langle df_2, J_3 df_1 \rangle + \langle df_3, J_1 df_2 \rangle + \langle df_1, J_2 df_3 \rangle \big)
\end{equation}
\end{prop}
\begin{proof} Using the assumption that $(g, J_1, J_2, J_3, \omega_1, \omega_2, \omega_3)$ is hyperK\"ahler 
\begin{equation} \label{naom-T4-exp}
    \nabla \omega = df_1 \otimes \omega_1 + df_2 \otimes \omega_2 + df_3 \otimes \omega_3 \; .
\end{equation} 
From this
$$ \delta \omega = -J_1 df_1 - J_2 df_2  -J_3 df_3 \; ,$$
and as $\theta = J \delta \omega$, where $J = f_1 J_1 + f_2 J_2 + f_3 J_3$, the formula (\ref{theta-T4-exp}) follows. 
Differentiating (\ref{theta-T4-exp}) and using one more time that all $J_i$ are parallel, we get:
$$ (\nabla_X \theta)(Y) = - df_2(X) \, df_1(J_3 Y) - f_2 (\nabla^2_{X J_3 Y} f_1) +  df_1(X) \, df_2 (J_3 Y) + f_1 (\nabla^2_{X J_3 Y} f_2) +$$ 
$$ \hspace{1cm} + {csum (1, 2, 3)} \; ,$$
where $csum(1, 2, 3)$ denotes the cyclic sum of the previous line. Now formula (\ref{deltheta-T4-exp}) follows noting that when taking trace in $X,Y$ the second and fourth terms on the first line vanish as the Hessian of $f_1$ (or of $f_2$) is symmetric, so its inner product with $\omega_3$ is zero. The expression (\ref{N-T4-exp}) for the Nijenhuis tensor is obtained from the formula (\ref{naom-T4-exp}) for $\nabla \omega$ and the relation (\ref{nablaom-dim4}). Indeed, from (\ref{nablaom-dim4}), one immediately obtains
$$N_X = J(\nabla_X \omega) - (\nabla_{JX} \omega) \; ,$$
and now one just uses (\ref{naom-T4-exp}).
\end{proof}

Note that if one of the functions $f_1, f_2, f_3$ is a constant, it follows that $\delta \theta = 0$, but $\delta \theta \neq 0$ is the generic case when all three functions are non-constant.

\vspace{0.3cm}

For the rest of this section, assume that $s\neq 0$ everywhere on $M$, so that Proposition \ref{ah1-Gaud} holds, hence, we have that the $\mathcal{AH}_1$ structure $(g, J, \omega)$ is Gauduchon, that is, $\delta \theta = 0$. Denote by $\mathcal{U}$ the set of points where $N\neq 0$. Note that on $\mathcal{U}$, $\theta \neq 0$, while on the complement of $\mathcal{U}$, both $N$ and $\theta$ vanish. This is true because of relation (\ref{kappa-s}), the $\mathcal{AH}_1$ assumption, and $\delta \theta = 0$. The following definitions make sense on the set $\mathcal{U}$, but we will be able to obtain some global results on $M$.

%
%

On the set $\mathcal{U}$, define the unit vector $ T = \theta^{\sharp}/ |\theta|$. From the Proposition \ref{ah1-facts} (c), ${\rm Span}(T, JT)$ is orthogonal to ${\rm Image}~N $. The endomorphism $L : {\rm Image } \; N \rightarrow {\rm Image } \; N $ defined by $L(U) = N(T, U)$ is symmetric (with respect to the inner product induced on ${\rm Image} \; N $ by the metric) and has zero trace. Let $V$ be a unit eigenvector of $L$, corresponding to the positive eigenvalue $\lambda$. Automatically, $-JV$ is an eigenvector of $L$, with eigenvalue $-\lambda$. With respect to the basis $\{ T, JT, V, JV \}$, the Nijenhuis tensor has the form
\begin{equation} \label{N-specframe}
 N = \lambda V \otimes ( T \wedge V - JT \wedge JV ) - \lambda JV \otimes ( JT \wedge V + T \wedge JV ) \; .
\end{equation}
Note that 
\begin{equation} \label{lambsq}
    \lambda^2 = \frac{1}{4} |N|^2 = |\theta|^2 \; .
\end{equation}
where for the last equality, we used $\delta \theta = 0$, obtained in Proposition \ref{ah1-Gaud}.

Consider the frame $\{ \psi, J\psi \}$ of $\leftr \La^{0,2}M \rightr$ which is ``aligned'' with the Nijenhius tensor; that is, let 
\begin{equation} \label{psiJpsi}
\psi = T \wedge V - JT \wedge JV \; , J\psi = JT \wedge V + T \wedge JV \; .
\end{equation}
Let $\tilde{a}, \tilde{b}, \tilde{c} $ and $\tilde{n}$ be the 1-forms corresponding to relations (\ref{abc}) and (\ref{NphiJphi}) with respect to the frame $\{ \psi, J\psi \}$. From (\ref{ab-thetan}), we directly get
$$ \tilde{a} - J \tilde{b} = J\psi(\theta) = |\theta| JV \; , \; \; \tilde{a} + J \tilde{b} =  \tilde{n} = - \lambda JV \; .  $$
From these, we easily get
\begin{equation} \label{tatb}
\tilde{a} = \frac{1}{2} (|\theta| - \lambda) JV \; , \; \; \tilde{b} = - \frac{1}{2} (|\theta| + \lambda) V \; , \; \; \tilde{n} = - \lambda JV \; .   
\end{equation}
By (\ref{lambsq}) and as $\lambda >0$, under our ${\mathcal AH}_1$ conclusions so far, including $\delta \theta =0$ from Proposition \ref{ah1-Gaud}, we get
\begin{equation} \label{tatbah1}
\tilde{a} = 0 \; , \; \; \tilde{b} = - |\theta| \, V \; .   
\end{equation}
Relations (iii) of Proposition \ref{ah1-locchar}  become
$$ 0 = \tilde{c} \wedge \tilde{b} \; , \; \; d \tilde{b} = 0 \; . $$
The first of these implies $\tilde{c} = \mu \, \tilde{b} $, for some function $\mu$.
But note that 
$$- d \tilde{c}= R(\omega) +  \tilde{a} \wedge \tilde{b}$$
is the Chern form $\gamma$. As we already know that $0 =\Phi = \tilde{a} \wedge \tilde{b}$, it follows that
$$ \rho = \rho^* = -d(\mu \, \tilde{b} ) = - d \mu \wedge \tilde{b} \; ,$$
where for the first equal sign we used the ${\mathcal AH}_1$ assumption, while for the last equal sign we used $d \tilde{b} = 0$ (again a consequence of ${\mathcal AH}_1$). Thus, our future goal is to prove that $\mu$ is a constant.

In order to do this, we exploit $ d \tilde{b} = 0$. Note that $\tilde{b} = - \psi(\theta) $, so taking the covariant derivative of this, one has
\begin{align*}
    (\nabla_A \tilde{b})(B) &= - (\nabla_A \psi)(\theta)(B) - \psi (\nabla_A \theta)(B) \\
                            &=  - (\tilde{c} (A) J\psi )(\theta)(B) + (\nabla_A \theta)(\psi B) \\
                            &= - \tilde{c} (A)  |\theta| JV (B) + (\nabla_A \theta)(\psi B) \\
                            &= \tilde{c} (A) J \tilde{b}(B) + (\nabla_A \theta)(\psi B) \; .
\end{align*}
Skew-symmetrizing in $A, B$, we thus get:
$$ 0 = d \tilde{b}(A,B) = (\tilde{c} \wedge J \tilde{b}) (A,B) + (\nabla_A \theta)(\psi B) - (\nabla_B \theta)(\psi A) \; ,$$
and using that $\tilde{c} = \mu \; \tilde{b}$, the above is rewritten as:
\begin{equation} \label{nathetapsi1}
    (\nabla_A \theta)(\psi B) - (\nabla_B \theta)(\psi A) = - \mu \, (\tilde{b} \wedge J\tilde{b})(A,B) \; .
\end{equation}
Replacing $A$ by $\psi A$ and $B$ by $\psi B$, we obtain
\begin{equation} \label{nathetapsi2}
   - (\nabla_{\psi A} \theta)(B) + (\nabla_{\psi B} \theta)(A) =  \mu \, (\psi \tilde{b} \wedge J\psi\tilde{b})(A,B) \; .
\end{equation}
Subtracting relations (\ref{nathetapsi1}) and (\ref{nathetapsi2}), and using that $\tilde{b} = - |\theta| V$ and the specific definition of the gauge $\psi$, we get
$$ 2 (d\theta)^{\psi - anti}(\psi A, B) = - \mu |\theta|^2 \omega(A, B) \; .$$
This is easily seen as equivalent with
$$ 2 (d\theta)^{\psi - anti} = \frac{\mu |\theta|^2}{2} J\psi \; . $$
The $\omega$-component of the above (note that $\omega$ is $\psi$-anti-invariant) yields the known fact that $\langle d\theta ,\omega \rangle = 0$, but the $J\psi$-component captures the new information 
\begin{equation} \label{dthetaJpsi}
    \langle d\theta , J\psi \rangle = \mu |\theta|^2  \; .
\end{equation}
We get the following partial result:
\begin{prop}
    Let $(M^4, g, J, \omega)$ be a compact almost Hermitian 4-manifold of class $\mathcal{AH}_1$. Assume also that $d \theta = 0$ (i.e. that the structure is locally conformally almost K\"ahler) and that $s\neq 0$ everywhere on $M$. Then $(M^4, g, J, \omega)$ must be a K\"ahler surface.
\end{prop}
\begin{proof}
    By contradiction, suppose that the manifold is not K\"ahler. Then the open set $\mathcal{U}$, where both $N$ and $\theta$ do not vanish is non-empty. On the other hand, the assumption $d \theta = 0$ and the relation (\ref{dthetaJpsi}) imply that $\mu = 0$ on $\mathcal{U}$. This further implies $\tilde{c} = 0$, so $\rho = d \tilde{c} = 0$ on $\mathcal{U}$. In particular, $ s = 0$ on $\mathcal{U}$, but this contradicts the assumption that $s \neq 0$ on $M$.
\end{proof}

\subsection{$\mathcal{H}_1$-surfaces} Suppose $(M^4, g, J, \omega)$ is a Hermitian surface satisfying the condition $(G_1)$. As noted already, in the compact case, we know that $\mathcal{H}_1 = \mathcal{K} $. We are next interested in the local aspect of the problem and the Lie algebra case. 

\vspace{0.2cm}

\noindent
For a Hermitian surface, the condition (\ref{invAH1-v2}) can be rewritten as
\begin{equation} \label{H1}
0 = Y \wedge JP(X) + JY \wedge P(X) - X \wedge JP(Y) - JX \wedge P(Y) \; ,  
\end{equation}
where
$$P(X) = \frac{1}{2} \theta(X) \theta + \nabla_X \theta - \frac{1}{4} |\theta|^2 X \; .$$
Alternatively, condition (\ref{H1}) can be seen to be equivalent with
$$\langle P(X) , \phi(Y) \rangle - \langle P(Y) , \phi(X) \rangle = 0 \; ,$$
for any gauge $\phi$, or further, with
\begin{equation} \label{H1phi}
(\nabla_X \theta)(\phi Y) - (\nabla_Y \theta)(\phi X) = \frac{1}{2} (\theta \wedge \phi(\theta) )(X,Y) - \frac{1}{2} |\theta|^2 \phi(X,Y) \; ,
\end{equation}
for any gauge $\phi$. For an orthonormal basis $\{e_i \}$, taking $X=e_i$ and $Y = \phi(e_i)$ and summing in the above relation, we get
$$2\delta \theta = - |\theta|^2 \; ,$$
that is, the relation (\ref{kappa-s}) for $\mathcal{H}_1$ 4-manifolds.
In (\ref{H1}), replacing $X,Y$ by $\phi(X), \phi(Y)$ and subtracting the two relations, we get $\langle d \theta , J\phi \rangle = 0$, for any gauge $\phi$. As $\langle d\theta \, , \, \omega \rangle = 0$ always holds, this shows that relation (\ref{H1}) implies that $d \theta \in \Lambda^{-}M$. On the other hand, if $\alpha \in \Lambda^{-} M$, we have 
$$ \langle \alpha(X) , \phi(Y) \rangle  - \langle \alpha(X) , \phi(Y) \rangle = 0 \; ,  $$
because self-dual and anti-self-dual forms commute as endomorphisms of $TM$.
Therefore, relation (\ref{H1}) is equivalent to satisfying all three of the following
\begin{equation} \label{H11and2}
   2\delta \theta = - |\theta|^2 \; , \; \; d \theta \in \Lambda^{-}M \; , \mbox{ and } 
\end{equation}
\begin{equation} \label{H1phisym}
(\nabla \theta)_0^{sym}(X, \phi Y) - (\nabla \theta)_0^{sym}(Y,\phi X) = \frac{1}{2} (\theta \wedge \phi(\theta) )(X,Y) - \frac{1}{4} |\theta|^2 \phi(X,Y) \; ,
\end{equation}
with (\ref{H1phisym}) valid for any gauge $\phi$. Replacing $X$ by $\phi(X)$ in (\ref{H1phisym}), and eliminating the arguments $X,Y$ we get
\begin{equation} \label{H1phisym2} \nonumber
2\big( (\nabla \theta)_0^{sym} \big)^{\phi-inv} = \frac{1}{4} |\theta|^2 g - \frac{1}{2} \big( \theta \otimes \theta + \phi(\theta) \otimes \phi(\theta) \big) \; ,
\end{equation}
for any gauge $\phi$. Next, replace the gauge $\phi$ by $J\phi$ in the above relation to also get
\begin{equation} \label{H1Jphi-inv} \nonumber
2\big( (\nabla \theta)_0^{sym} \big)^{J\phi-inv} = \frac{1}{4} |\theta|^2 g - \frac{1}{2} \big( \theta \otimes \theta + J\phi(\theta) \otimes J\phi(\theta) \big) \; .
\end{equation}
Subtracting the last two relations, we get an expression for the $J$-anti-invariant part of $(\nabla \theta)_0^{sym}$, which is independent of $\phi$
\begin{equation} \label{H1J-anti}
\big( (\nabla \theta)^{sym} \big) '' = - \frac{1}{4} \big( \theta \otimes \theta - J\theta \otimes J \theta \big) \; .
\end{equation}
Note that because $d\theta = (\nabla \theta )^{skew}$, we also know that 
$ \big( (\nabla \theta)^{skew} \big) '' = 0 \; .$
Thus, we get the following (local) characterization of Hermitian surfaces satisfying Gray's first condition.
\begin{prop} \label{H1-loc-theta}
Let $(M^4, g, J, \omega)$ a Hermitian surface with Lee form $\theta$. Then the manifold satisfies Gray's first condition if and only if
\begin{equation} \label{H1-loc-2cond}
   2\delta \theta = - |\theta|^2 \; ,  \mbox{ and } \; (\nabla \theta)'' = - \frac{1}{4} \big( \theta \otimes \theta - J\theta \otimes J \theta \big) \; . 
\end{equation}  
\end{prop}
\begin{proof} In the lines preceding the statement we indicated the proof for the forward implication. We now give a few more details on showing the reverse statement. Assume that the relations in (\ref{H1-loc-2cond}) are satisfied. Because $\nabla \theta = (\nabla \theta)^{sym} + (\nabla \theta)^{skew} $ and the right side of the second relation in (\ref{H1-loc-2cond}) is a symmetric tensor, we conclude 
$$ (d\theta)'' = \big( (\nabla \theta)^{skew} \big)'' = 0 \; ,$$
hence $d \theta \in \Lambda^{-}M$ (as $\langle d\theta \, , \, \omega \rangle = 0$ is automatic). It is clear that the second relation of (\ref{H1-loc-2cond}) implies that the relation (\ref{H1J-anti}) holds. Combined with the first part of (\ref{H1-loc-2cond}), relation (\ref{H1J-anti}) further implies that the trace-free symmetric part of $\nabla \theta$ is given by 
\begin{equation} \label{A-def}
       (\nabla \theta)_0^{sym} = \frac{1}{8} |\theta|^2 g - \frac{1}{2} \theta \otimes \theta + A \; , 
\end{equation}
where $A$ is some $J$-invariant trace-free symmetric tensor. Note that $A$ must automatically be $\phi$-anti-invariant for any gauge $\phi$. Indeed, if $X$ is an eigenvector of $A$ corresponding to an eigenvalue $a$, and $\phi$ is an arbitrary gauge, then
$$ A = a \big( X \otimes X + JX \otimes JX \big) - a \big( \phi(X) \otimes \phi(X) + J\phi(X) \otimes J\phi(X) \big) \; .$$
Taking the $\phi$-invariant part of (\ref{A-def}), we see that condition (\ref{H1phisym2}) is satisfied. Thus, we showed that all the relations in (\ref{H11and2}) and (\ref{H1phisym}) hold for an arbitrary gauge $\phi$, hence Gray's first condition is satisfied.
\end{proof}


\vspace{0.2cm}

\noindent
The condition is even more specific in the case of a Lie algebra.

\begin{prop} \label{H1Liealg}
    Let $\mathfrak{g}$ be a 4-dimensional Lie algebra equipped with a Hermitian structure $(g,J,\omega)$ ($J$ integrable). Then Gray's first condition holds if and only if  
    \begin{equation} \label{H1-loc-LA}
   \nabla \theta = \frac{1}{2} J\theta \otimes J\theta +(x\theta +yJ\theta) \otimes V + J(x\theta +yJ\theta) \otimes JV\; , 
\end{equation} 
where $x, y$ are constants and $V$ is a unit vector orthogonal to Span$(\theta, J\theta)$.
In particular,
\begin{equation} \label{dtheta-H1-LA}
   d \theta = (x\theta +yJ\theta) \wedge V + J(x\theta +yJ\theta) \wedge JV \; , 
\end{equation} 
\begin{equation} \label{dJtheta-H1-LA}
    d (J\theta) = \frac{1}{2} |\theta|^2(T \wedge JT - 2 V \wedge JV) +(x\theta +yJ\theta) \wedge JV - J(x\theta +yJ\theta) \wedge V \; .
\end{equation}  
\end{prop}
\begin{proof} If relation (\ref{H1-loc-LA}) is satisfied, it is easy to check that both relations in (\ref{H1-loc-2cond}) hold, hence the structure satisfies Gray's first condition by Proposition \ref{H1-loc-theta}. Assume next that the Lie algebra satisfies Gray's first condition, and moreover assume $\theta \neq 0$, as otherwise all relations claimed in the statement obviously hold. Let $T = \theta^{\sharp}/|\theta|$ as before, and let $V$ be any unit vector orthogonal to 
Span$(T, JT)$. Let $m, p, q$ be 1-forms so that
$$ \nabla T = m \otimes JT + p \otimes V + q \otimes JV \; .$$
Starting from this, a short computation implies that
$$ (\nabla T)'' = \frac{1}{2} \Big[ m \otimes JT + Jm \otimes T + (p +Jq) \otimes V -J(p+Jq) \otimes JV\Big] \; .$$
Comparing the above with the relation for $(\nabla T)''$ that follows from Proposition \ref{H1-loc-theta}
$$ (\nabla T)'' = -\frac{1}{4} |\theta| \big( T \otimes T - JT \otimes JT \big) \; , $$
we obtain
$$ m = \frac{1}{2}|\theta| JT \; , \; \; q =Jp \; .$$
We next replace these in the initial expression for $\nabla T$ and multiply by $|\theta|$ to get 
$$\nabla \theta = \frac{1}{2} J\theta \otimes J\theta + |\theta| (p \otimes V + Jp \otimes JV) \; .$$
From this, it follows that
$$\delta \theta = -\frac{1}{2} |\theta|^2 - 2 |\theta| p(V) \; ,$$
thus, by first relation in (\ref{H1-loc-2cond}), it follows that $p(V) = 0$.
Next, skew-symmetrize the above expression for $\nabla \theta$ to obtain
$$ d \theta = |\theta| ( p \wedge V + Jp \wedge JV) \; .$$
It is clear that the condition $(d \theta)'' = 0$ is satisfied, but the condition $\langle d\theta , \omega \rangle= 0$ yields $p(JV) = 0$.
Therefore, there are constants $x, y$ so that $ p =x T + yJT$, so formulas 
(\ref{H1-loc-LA}) and (\ref{dtheta-H1-LA}) have been proved. Formula (\ref{dJtheta-H1-LA}) also follows after a short computation first for $\nabla J\theta$ and then skew-symmetrizing.
\end{proof}

\noindent
With this proposition in hand, we next show that for Hermitian Lie algebras the equality $\mathcal{H}_1 = \mathcal{K} $ must hold.

\begin{prop} \label{H1=K-LA}
Let $\mathfrak{g}$ be a 4-dimensional Lie algebra equipped with a Hermitian structure $(g,J,\omega)$ which satisfies Gray's first condition. Then $(g,J,\omega)$ must be a K\"ahler structure.     
\end{prop}
\begin{proof} If the Lie algebra is unimodular, then the Lee form $\theta$ must be co-closed, i.e. $\delta \theta = 0$, so the conclusion follows immediately from the first relation of (\ref{H1-loc-2cond}). The main effort, thus, will be for the non-unimodular case. From the previous proposition, we have that relations (\ref{H1-loc-LA}), (\ref{dtheta-H1-LA}) and 
(\ref{dJtheta-H1-LA}) hold for some constants $x, y$. 

Consider first the case $x=y=0$, i.e. the case $d\theta = 0$. As we know that $V \wedge JV$ is a closed form, the formula (\ref{dJtheta-H1-LA}) with $x=y=0$ implies that $T \wedge JT$ is also closed. Thus, we get
$$ d\omega = d(T \wedge JT + V \wedge JV) = 0, $$
hence $(g, J, \omega)$ is a K\"ahler structure. 

Therefore, from now on, assume that at least one of the constants $x, y$ is not zero. Then formulas (\ref{dtheta-H1-LA}) and (\ref{dJtheta-H1-LA}) yield two symplectic forms on the Lie algebra, compatible with the opposite orientation. Indeed, let
$$ \bar{\omega}_1 = d\theta \; , \; \; \bar{\omega}_2 = d(J\theta) + \frac{1}{2} |\theta|^2 V \wedge JV \; .$$
Observe that $\bar{\omega}_1, \bar{\omega}_2 \in \Lambda^-$. Also note that $\bar{\omega}_1 \wedge \bar{\omega}_2 = 0 $. Thus, after a rescaling so that 
$$ \bar{\omega}_1^2 = \bar{\omega}_2^2 = - \omega^2 \; , $$
it follows that the Lie algebra with the given metric also admits, with the opposite orientation, a compatible complex symplectic structure. In this structure, at least one of the forms, $\bar{\omega}_1$, to be precise, is exact. An examination of the table of Corollary 4.2 in \cite{Ov} shows that only the Lie algebra $\mathfrak{r}_2' = \mathfrak{aff}(\mathbb{C})$ admits such complex symplectic structures. This Lie algebra also admits complex structures compatible with either orientation which are nicely described in the table associated with Proposition 3.2 of the same paper \cite{Ov} of Ovando. Of these complex structures, it turns out that only the one corresponding to the parameter $b_1 = -i$ could be compatible with the complex symplectic structure with the opposite orientation. Namely, we want that $\leftr \La^{0,2}M \rightr$ be orthogonal to both forms of the complex symplectic structure. This happens only for $b_1 = -i$. A direct check will show that this complex structure will not yield a structure satisfying $(G_1)$.
\end{proof}

\section{A unique non-K\"ahler 4-dimensional $\mathcal{AH}_1$ Lie algebra}

At the start of this section, we describe the Lie algebra $\mathcal{A}_{3,6}\oplus \mathcal{A}_1$ (here we use the same notation of Lie algebras as~\cite{Pat}). The structure of the Lie algebra is 
$$[e_1,e_3]=-e_2,\quad [e_2,e_3]=e_1,$$
where $\{e_1,e_2,e_3,e_4\}$ is a basis of the Lie algebra.
The associated simply connected group to the Lie algebra $\mathcal{A}_{3,6}\oplus \mathcal{A}_1$ admits lattices~\cite{B}. If 
$\{e^1,e^2,e^3,e^4\}$ is the dual basis, the structure equations can be rewritten as
$$ d e^1 = - e^2 \wedge e^3 \; , \; \; d e^2 =  e^1 \wedge e^3\; , \; \; d e^3 = 0 \; , \; \; d e^4 = 0 \; .$$
It can be checked that the Riemannian metric 
$$g=\sum_{i=1}^4e^i\otimes e^i \; ,$$ 
is a flat metric. By results of Milnor, \cite{Mil}, it follows that $\mathcal{A}_{3,6}\oplus \mathcal{A}_1$ is, in fact, the only 4-dimensional non-abelian Lie algebra with a flat metric. It turns out that with this flat metric and any of the orientations, the Lie algebra admits a compatible K\"ahler structure. Indeed, it is easy to check that both $J_0$ and $\bar{J}_0$, defined by
$$ J_0 e_1 = e_2 \; , \; J_0 e_3 = -e_4 \; , \; \; \; \bar{J}_0 e_1 = e_2 \; , \; \bar{J}_0 e_3 = e_4 \; $$
are integrable almost complex structures and the induced 2-forms
$$\omega_0 = e^1 \wedge e^2 - e^3 \wedge e^4 \; , \; \; \bar{\omega}_0 = e^1 \wedge e^2 + e^3 \wedge e^4 $$
are both closed. However, with either orientation, any other left-invariant $g$-compatible almost complex structure $J$ on the Lie algebra $\mathcal{A}_{3,6}\oplus \mathcal{A}_1$ is non-K\"ahler and $(g,J)$ will be a basic example
of an $\mathcal{AH}_1$-structure, as described in Remark \ref{basic-exp}. For example, we can consider the almost-Hermitian structure $(g,J)$, where $J$ is defined by
$$Je_1=e_3,\quad Je_2=e_4.$$
As observed in \cite{BL}, $(g,J)$ induces a locally conformally symplectic structure on $\mathcal{A}_{3,6}\oplus \mathcal{A}_1$, i.e. $d\theta = 0$. 


\vspace{0.2cm}

\noindent
The goal of this section is to prove the following result:
\begin{theorem} \label{uniqueLA}
The only 4-dimensional Lie algebra admitting a non-K\"ahler left-invariant $\mathcal{AH}_1$-structure is the Lie algebra $\mathcal{A}_{3,6}\oplus \mathcal{A}_1$ endowed with a flat Riemannian metric and a (non-K\"ahler) almost complex structure compatible with this metric.
\end{theorem}

\begin{proof} The case of a Hermitian Lie algebra has been solved by Proposition \ref{H1=K-LA}, so we assume $N\neq 0$ from now on.

We first prove the result in the case when the Lie algebra is unimodular. Let $\mathfrak{g}$ be a 4-dimensional unimodular Lie algebra admitting a non-K\"ahler $\mathcal{AH}_1$-structure $(g, J, \omega)$. Because of the unimodularity assumption, the Lee form $\theta$ must be co-closed, i.e. $\delta \theta = 0$. From $\kappa -s = 0$ and the relation (\ref{kappa-s}), we have 
 $|\theta|^2 = \frac{1}{4} |N|^2 \neq 0$. Therefore, the frame $\{T, JT, V, JV\}$ for $TM$ from the previous section is well-defined on $\mathfrak{g}$ and so is the ``Nijenhuis alligned'' frame $\{ \psi, J\psi \}$ from  (\ref{psiJpsi}). If $\tilde{a}, \tilde{b}, \tilde{c} $ are the 1-forms corresponding to the relations (\ref{abc}) with respect to the frame $\{ \psi, J\psi \}$, as we have shown in the previous section, $\tilde{a}=0$, $d\tilde{b} = 0$, $\tilde{c} = \mu \tilde{b}$. As we are on a Lie algebra with an invariant structure, $\mu$ is a constant, so $\gamma^J = - d\tilde{c} = \mu d\tilde{b} = 0$. As the form $\Phi$ vanishes (from Proposition \ref{ah1-facts}), it follows that $\rho^* = \rho = 0$. Thus, the Lie algebra must be Ricci flat and ASD (because of the $\mathcal{AH}_1$-assumption).
 
 Next, using the classification of 4-dimensional Lie algebras that admit an ASD metric by De Smedt and Salamon, \cite{DeSmSa}, one concludes that a Ricci flat ASD metric on a Lie algebra must actually be flat. From Milnor, \cite{Mil}, it follows that the only 4-dimensional non-abelian Lie algebra with a flat metric is $\mathcal{A}_{3,6}\oplus \mathcal{A}_1$. As we observed already, this Lie algebra admits a compatible K\"ahler structure $(g, J_0)$ (with respect to either orientation), but it admits many compatible non-K\"ahler almost complex structures $(g,J)$ as well. As the metric is flat, any one of these is trivially of class $\mathcal{AH}_1$.  

 Next, we consider the case when $\mathfrak{g}$ is a 4-dimensional non-unimodular Lie algebra admitting an $\mathcal{AH}_1$-structure $(g, J, \omega)$ with $N\neq 0$. Let us assume that $\delta \theta $ is a non-zero constant, something which is apriori possible in the non-unimodular case.
 Choose the special frame $\{ T, JT, V, JV\}$ so that the Nijenhuis tensor has the form (\ref{N-specframe}) and choose the aligned gauge $\{ \psi, J \psi \}$ as in (\ref{psiJpsi}). Let $\tilde{a}, \tilde{b}, \tilde{c} $ be the 1-forms corresponding to relations (\ref{abc}) with respect to the frame $\{ \psi, J\psi \}$ and recall that $\tilde{a}$ and $\tilde{b}$ are given by (\ref{tatb}). Replacing $\tilde{a}$ and $\tilde{b}$ in the $\mathcal{AH}_1$ conditions gives
 $$ d \tilde{a} = \tilde{c} \wedge \tilde{b} \; , \; \; d \tilde{b} = -\tilde{c} \wedge \tilde{a} \; ,$$
 and these can be seen to be equivalent with
 \begin{equation} \label{dJV-dV}
 d(JV) = -\Big(\frac{|\theta| + \lambda}{|\theta| - \lambda} \Big) \, \tilde{c} \wedge V \; , \; \; d(V) = \Big(\frac{|\theta| - \lambda}{|\theta| + \lambda} \Big) \, \tilde{c} \wedge JV \; .     
 \end{equation}
 A direct computation (see also Lemma 2.2 in \cite{DS}) shows that for any 1-form $\tau$
 $$ (dJ\tau)'' = \frac{1}{2} N_{J\tau} + J(d\tau)'' \; .$$
 Applying this with $\tau = V$ and using (\ref{dJV-dV}), a short computation yields
 $$\lambda(JT \wedge V + T \wedge JV) = \frac{4|\theta| \cdot \lambda}{|\theta|^2 - \lambda^2} \big(\tilde{c} \wedge V - J \tilde{c} \wedge JV \big) \; .$$
 We have $\lambda \neq 0$ by the assumption that $N \neq 0$, thus the above relation implies
 \begin{equation} \label{tildecT} \nonumber
 \tilde{c}(T) = 0 \; , \; \; \tilde{c}(JT) = \frac{|\theta|^2 - \lambda^2}{4|\theta|} \; .
 \end{equation}
Therefore, we get the following expression for the 1-form $\tilde{c}$
\begin{equation} \label{tc} \nonumber
\tilde{c} = \frac{|\theta|^2 - \lambda^2}{4|\theta|} \, JT + \tilde{c}(V) \, V + \tilde{c}(JV) \, JV \; .    
\end{equation}
Differentiating the above relation and using (\ref{dJV-dV}), we eventually obtain
\begin{eqnarray} \label{dtc} \nonumber
d\tilde{c} &=& \frac{|\theta|^2 - \lambda^2}{4|\theta|} \, d(JT) + \Big[\tilde{c}(V)^2\Big(\frac{|\theta| - \lambda}{|\theta| + \lambda} \Big) + \tilde{c}(JV)^2\Big(\frac{|\theta| + \lambda}{|\theta| - \lambda} \Big) \Big] \, V \wedge JV + \\ \nonumber
& & + \tilde{c}(V) \frac{(|\theta| - \lambda)^2}{ 4|\theta| } \, JT \wedge JV - \tilde{c}(JV) \frac{(|\theta| + \lambda)^2}{ 4|\theta| } \, JT \wedge V \; .    
\end{eqnarray}
On the other hand, we know by relation (\ref{gamma=-dc}) and Proposition \ref{ah1-facts} that
$$ d\tilde{c} = - \gamma = k \, V \wedge JV \; $$
for some constant $k$. Combining these two expressions for $ d\tilde{c}$, we get
\begin{equation} \label{dJT}
d(JT) = - \tilde{c}(V) \Big(\frac{|\theta| - \lambda}{|\theta| + \lambda} \Big) \, JT \wedge JV + \tilde{c}(JV) \Big(\frac{|\theta| + \lambda}{|\theta| - \lambda} \Big) \, JT \wedge V + l \, V \wedge JV \; ,   
\end{equation}
for some constant $l$.
Next, we compute $d\omega$ in two different ways and extract information from this. On the one hand,
\begin{equation} \label{dom1}
   d\omega = \theta \wedge \omega = |\theta| \, T \wedge (T \wedge JT + V \wedge JV) = |\theta| \, T \wedge V \wedge JV \; . 
\end{equation}  
On the other hand,
$$ d\omega = d(T\wedge JT) + d(V \wedge JV) = dT \wedge JT - T \wedge d(JT) \; .$$
where for the second equal sign we used $d(V \wedge JV) = 0$ from Proposition \ref{ah1-facts}. Replacing $d(JT)$ from (\ref{dJT}) in the above, we get 
\begin{eqnarray} \label{dom2}
  & & d\omega = dT \wedge JT - \\ \nonumber
   &-& T \wedge \Big[ - \tilde{c}(V) \Big(\frac{|\theta| - \lambda}{|\theta| + \lambda} \Big) \, JT \wedge JV + \tilde{c}(JV) \Big(\frac{|\theta| + \lambda}{|\theta| - \lambda} \Big) \, JT \wedge V + l \, V \wedge JV \Big] \; . 
\end{eqnarray}
Comparison of the coefficients of the terms in $T \wedge V \wedge JV$ in (\ref{dom1}) and (\ref{dom2}), yields $l = - |\theta|$, which, by relation (\ref{dJT}), is equivalent to (in the Lie algebra setting)
$$ \langle d(J\theta) , \omega \rangle = - |\theta|^2 \; .$$
However, on any almost Hermitian 4-manifold (e.g. see Proposition 2.3 in \cite{DS})
$$ \langle d(J\theta) , \omega \rangle = - |\theta|^2 - \delta \theta \; .$$
Comparing the two relations above, it follows that $\delta \theta = 0$. But this is a contradiction with the assumption $\delta \theta \neq 0$ that we made at the outset of this computation. We conclude that for any non-unimodular Lie algebra of class $\mathcal{AH}_1$ with $N \neq 0$, we must have $\delta \theta = 0$. In this case, the argument from the unimodular case applies, and it implies that such a Lie algebra is flat. But by Milnor, the only 4-dimensional flat non-abelian Lie algebra is $\mathcal{A}_{3,6}\oplus \mathcal{A}_1$, which is unimodular. Thus, there is no non-unimodular Lie algebra of class $\mathcal{AH}_1$ with $N \neq 0$ and the proof is completed.

\end{proof}


\begin{thebibliography}{99}

\bibitem{AAD} V. Apostolov, J. Armstrong, T. Draghici, {\it Local
models and integrability of certain almost K\"ahler 4-manifolds},  Math. Ann. {\bf 323} (2002), 633-666.

\bibitem{AAD-AGAG} V. Apostolov, J. Armstrong, T. Draghici, {\it Local
rigidity of certain classes of almost K\"ahler 4-manifolds}, Ann. Glob. Anal. Geom. {\bf 21} (2002), 151-176.

\bibitem{AD-QJM} V. Apostolov, T. Draghici, {\it Almost K\"ahler 4-manifolds with $J$-invariant Ricci tensor and special Weyl tensor}, Quart. J. Math. 51 (2000), 275-294.

\bibitem{AG} V. Apostolov, P. Gauduchon, {\it The Riemannian Goldberg-Sachs theorem}, Int. J. Math. 8 (1997), 421-439.

\bibitem{BL} G. Barbaro, M. Lejmi, {\it Second-Chern-Einstein metrics on 4-dimensional almost-Hermitian manifolds}, Complex Manifolds 2023, https://doi.org/10.1515/coma-2022-0150

\bibitem{besse} A. L. Besse, Einstein manifolds, Ergeb. Math. Grenzgeb., Springer-Verlag, Berlin, Heidelberg, New York, 1987.

\bibitem{Bl} D. E. Blair, {\it The "total scalar curvature" as a
symplectic invariant and related results}, Proc. 3rd Congress of Geometry, Thessaloniki (1991),
79--83.

\bibitem{BI} D. E. Blair and S. Ianus, {\it Critical associated metrics on
symplectic manifolds},
Contemp. Math. {\bf 51} (1986), 23-29.

\bibitem{B} C. Bock, {\it On low-dimensional solvmanifolds}, Asian J. Math., 20(2):199--262, 2016.

\bibitem{DM} J. Davidov and O. Mu\u{s}karov, {\it Twistor spaces with Hermitian Ricci tensor}, Proc. Amer. Math. Soc. {\bf 109} (1990), 1115--1120.

\bibitem{DeSmSa} V. De Smedt, S. Salamon, {\it Anti-self-dual metrics on Lie groups}, Contemp. Math. 308, AMS, Providence, RI, 2002, 63–75.

\bibitem{DS} T. Draghici, C. Sayar, {\it Some remarks on almost Hermitian functionals}, Ann. Global Analysis and Geom. (2024), https://doi.org/10.1007/s10455-023-09943-8

\bibitem{gauduchon1} P. Gauduchon, {\it Hermitian connections and Dirac operators}, Boll. U.M.I. (7) {\bf 11}-B (1997), Suppl. fasc. 2,
257--288.

\bibitem{Go} S. I. Goldberg, {\it Integrability of almost \ka manifolds},
Proc. Amer. Math. Soc. {\bf 21} (1969), 96--100.

\bibitem{Gr} A. Gray, {\it Curvature identities for Hermitian and almost
Hermitian manifolds}, T{\^o}hoku Math. J. {\bf 28} (1976), 601--612.

\bibitem{GrHe} A. Gray, L.M. Hervella, {\it The sixteen classes of almost Hermitian manifolds and their linear invariants}, Ann. di Mat. Pura ed Appl. 123 (1980), 35--58.

\bibitem{Kok}  G. Kokarev, {\it On pseudo-harmonic maps in conformal geometry}, Proc. London Math. Soc. 99 (2009), 168--194.

\bibitem{Mil} J. Milnor, {\it  Curvature of left invariant metrics on Lie groups}, Adv. Math. 21, 293–329 (1976).

\bibitem{MM} A. Moroianu, S. Moroianu, {\it On pluricanonical locally conformally Kähler manifolds}, Int. Math. Res. Not. 2017 no. 14 (2017), 4398-4405.

\bibitem{Mu} O. Muskarov, {\it On Hermitian surfaces with $J$-invariant Ricci tensor}, J. Geom. 72 (2001), no. 1-2, 151–156.

\bibitem{Ov} G. Ovando, {\it Complex, symplectic and K\"ahler structures on four dimensional Lie groups}, Rev. Un. Mat. Argentina, 45, no. 2, 2004, 55-67.

\bibitem{Pat} J. Patera, R. T. Sharp, P. Winternitz, and H. Zassenhaus, {\it Invariants of real low dimension Lie algebras}, J.
Mathematical Phys., 17(6):986--994, 1976.

\bibitem{Sek87} K. Sekigawa, {\it On some 4-dimensional compact almost Hermitian manifolds}, J. Ramanujan Math. Soc. 2 (2) 1987, 101-116.


\bibitem{sekigawa} K. Sekigawa, {\it On some compact Einstein almost-K\"ahler
manifolds}, J. Math. Soc. Japan {\bf 36} (1987), 677--684.

\bibitem{TV} F. Tricerri and L. Vanhecke, {\it Curvature tensors on almost
Hermitian manifolds},
Trans. Amer. Math. Soc. {\bf 267} (1981), 365--398.


\end{thebibliography}
\end{document}